\documentclass[10pt,oneside,a4paper,reqno]{amsart}  %%(fold)
\usepackage{mathrsfs}
\usepackage{amssymb}

\usepackage{amssymb, bbm,
enumerate}
\usepackage{geometry}                % See geometry.pdf to learn the layout options. There are lots.
\usepackage[T1]{fontenc}

\usepackage{mathrsfs} 
\usepackage[shortlabels]{enumitem}

\usepackage[english]{babel}
\usepackage{graphicx}
\usepackage{color}
\usepackage[colorlinks=true, linkcolor=magenta, citecolor=cyan, urlcolor=blue]{hyperref}%link alle eq e ref
\usepackage{scalerel,stackengine}
\stackMath
\newcommand\reallywidehat[1]{%
\savestack{\tmpbox}{\stretchto{%
  \scaleto{%
    \scalerel*[\widthof{\ensuremath{#1}}]{\kern-.6pt\bigwedge\kern-.6pt}%
    {\rule[-\textheight/2]{1ex}{\textheight}}%WIDTH-LIMITED BIG WEDGE
  }{\textheight}% 
}{0.5ex}}%
\stackon[1pt]{#1}{\tmpbox}%
}
\renewcommand{\phi}{\varphi}

\renewcommand{\Re}{\textup{Re }}
\renewcommand{\Im}{\textup{Im }}

\newcommand{\meanv}[1]{\langle#1\rangle}

\newcommand{\mc}[1]{\mathcal{#1}}

\newcommand{\mb}[1]{\mathbb{#1}}

\DeclareMathOperator{\dive}{div}

\def\be{\begin{equation}}
\def\ee{\end{equation}}
\def\bea{\begin{eqnarray}}
\def\eea{\end{eqnarray}}
\def\ni{\noindent}
\def\nn{\nonumber}
\def\T{\mathbb{T}}

\def\C{\mathbb{C}}
\def\Z{\mathbb{Z}}
\def\N{\mathbb{N}}
\def\B{\mathscr{B}}

\def\Ga{\mathscr{G}}%GAUGE
\DeclareMathSymbol{\leqslant}{\mathalpha}{AMSa}{"36} % nicer `smaller or equal'
\DeclareMathSymbol{\geqslant}{\mathalpha}{AMSa}{"3E} % nicer `larger or equal'
\DeclareMathSymbol{\eset}{\mathalpha}{AMSb}{"3F}     % nicer `emptyset'
\renewcommand{\leq}{\;\leqslant\;}                   % redef. of < or =
\renewcommand{\geq}{\;\geqslant\;}                   % redef. of > or =

\DeclareMathOperator{\Id}{Id}

\DeclareMathOperator{\Span}{span}

\def\a{\alpha}
\def\e{\varepsilon}
\def\d{\delta}
\def\g{\gamma}

\def\D{\Delta}
\def\l{\lambda}
\def\r{\rho}
\def\s{\sigma}

\def\R{\mathbb{R}}
\def\C{\mathbb{C}}

\theoremstyle{plain}
\newtheorem{theorem}{Theorem}[section]
\newtheorem{lemma}[theorem]{Lemma}
\newtheorem{proposition}[theorem]{Proposition}
\newtheorem{corollary}[theorem]{Corollary}

\theoremstyle{definition}

\theoremstyle{remark}
\newtheorem{remark}[theorem]{Remark}

\numberwithin{equation}{section}

\definecolor{light}{gray}{.9}

\author{Giuseppe Genovese}
\address{Department Mathematik und Informatik, Universit\"at Basel Spiegelgasse 1, CH-4051 Basel, Switzerland}
\email{giuseppe.genovese@unibas.ch}

\author{Renato Luc\`a}
\address{BCAM - Basque Center for Applied Mathematics, 48009 Bilbao, Spain and Ikerbasque, Basque Foundation
for Science, 48011 Bilbao, Spain.}
\email{rluca@bcamath.org} 

\author{Nikolay Tzvetkov}
\address{Department of Mathematics (AGM), University of Cergy-Pontoise, 2, av. Adolphe Chauvin, 95302 Cergy-Pontoise Cedex, FRANCE}
\email{nikolay.tzvetkov@u-cergy.fr}

\title[Quasi-invariance of Gaussian Measures]
{Quasi-invariance of low regularity Gaussian measures under the gauge map of the periodic derivative NLS}

\date{\today}



\makeindex
\begin{document}
\begin{abstract}
The periodic DNLS gauge is an anticipative map with singular generator which revealed crucial in the study of the periodic derivative NLS. We prove quasi-invariance of the Gaussian measure on $L^2(\T)$ with covariance $[1+(-\D)^{s}]^{-1}$ under these transformations for any $s>\frac12$. This extends previous achievements by Nahmod, Ray-Bellet, Sheffield and Staffilani (2011) and Genovese, Luc\`a and Valeri (2018), who proved the result for integer values of the regularity parameter $s$.
\end{abstract}

\maketitle

\section{Introduction}

Let $\T:=\R /2\pi\Z$ and $\a \in \R$. We introduce $\Ga_\a \, : \, L^2(\mathbb{T}) \to L^2(\mathbb{T})$ to be defined by
\begin{equation}\label{gauge-change}
\Ga_\a (u)(x):= e^{ i \alpha \mathcal{I}[u(x)] }  u(x)\,,
\end{equation}
where   
\begin{equation}\label{DefMathcali}
 \mathcal{I}[u(x)] := \frac{1}{2\pi }\int_0^{2\pi }d\theta\int_{\theta}^x\Big(|u(y)|^2-\frac{\|u\|_{L^2(\T)}^2}{2\pi }\Big)dy \,.
\end{equation}
To get some insight, $\mathcal{I}[u(x)]$ can be though of as the unique periodic primitive of $|u|^2$ with zero average, therefore we can formally write
$$
\mathcal{I}[u(x)]=\partial_x^{-1}\Big(|u|^2-\frac{1}{2\pi}\|u\|_{L^2}^2\Big)\,. 
$$
We will much exploit this intuition in the sequel. 

An alternative formulation of (\ref{gauge-change}) is through the initial value problem 
\be\label{eq:gauge-eq}
\frac{d}{d\a}\Ga_{\a} u=i\mc I[\Ga_{\a} u]\Ga_{\a} u\,,\quad \Ga_{0} u=u\,.
\ee
The map $\alpha \to \Ga_\a$ is a one parameter group of transformations of $L^2(\T)$, in fact
\begin{equation}\label{eq:gauge-properties}
\Ga_0=\Id
\qquad
\text{and}
\qquad
\Ga_{\alpha_1}\circ\Ga_{\alpha_2}=\Ga_{\alpha_1+\alpha_2}\,,
\quad
\text{for any }\alpha_1,\alpha_2\in\R
\,. 
\end{equation}
This gauge was introduced in the periodic setting in \cite{Herr} in the context of the derivative nonlinear Sch\"odinger equation (DNLS). It has been conveniently used in different contexts regarding the DNLS: just to mention few examples, the study of the local well-posedness at low regularity is based on the use of such a gauge transformation \cite{Herr, GH, deng} and it revealed to be crucial also in the proof of the invariance of the Gibbs measures associated with the integrals of motions of DNLS \cite{NOR-BS12, GLV2}.

In this paper we investigate the way these maps transform the Gaussian measure on $L^{2}(\T)$ with covariance operator $[1+(-\D)^{s}]^{-1}$ for $s>\frac12$. Thanks to separability and the isomorphism between~$\C^{2N+1}$ and 
$$
E_N:=\Span_{\mb C}\{e^{inx}\,:\, |n|\leq N\}\,
$$
the space $L^2(\T)$ inherits the measurable-space structure by a standard limit procedure and we will denote by
$\mathscr{B}(L^2(\T))$ the Borel $\s$-algebra on~$L^{2}(\T)$.
We denote by $\g_s$ the Gaussian measure on $\B(L^2(\T)$ induced by the map
\begin{equation}\label{Def:gammaK}
\omega\longmapsto \sum_{n\in\Z} \frac{g_n(\omega)}{(1+|n|^{2s})^{\frac{1}{2}}}\, e^{inx} \, ,
\end{equation}
where $\{g_n(\omega)\}_{n\in\Z}$ are independent, identically distributed complex centred Gaussian random variables with unitary variance.
For any $s\in\R$ the triple $(L^2(\T), \B(L^2(\T)), \g_s)$ is a Gaussian probability space satisfying the 
concentration properties:
$$\g_s \bigg( \bigcap_{s' < s-\frac{1}{2}}  H^{s'}(\T) \bigg)=1 \, ,\quad \g_s \bigg(H^{s-\frac12}(\T) \bigg)=0\,. $$  
The $L^p$ spaces associated to $\g_s$ are denoted by $L^p(\g_s)$. 
For more details about this construction we refer for instance to \cite{kuo}.

Our main result is the quasi-invariance under the group $\{\Ga_\a\}_{\a\in\R}$ of $\g_s$ restricted to a ball in~$L^2(\T)$, of arbitrary size, 
defined by $B(R):=\{u\,:\,\|u\|_{L^2(\T)}\leq R\}$ for all $s>\frac12$.  Henceforth we set for brevity
\be\label{eq:def-gammatilde}
\tilde\g_s(A):=\g_s(A\cap\{ \|u\|_ {L^2}\leq R\})\qquad A\in\mathscr B(L^2(\T))\,.
\ee
%%%%%
\begin{theorem}\label{th:gauge}
Let $s>\frac12$, $R>0$. Then for every $\a\in\R$ there is $p_0>1$ and $\r_\a \in L^p(\tilde\g_s)$ for all $p\in[1,p_0)$ such that
\be\label{eq:THgauge}
(\tilde \g_s \circ \Ga_\a)(A)=\int_{A}\r_\a(u)\tilde \g_s(du) \qquad A\in\mathscr B(L^2(\T))\,.
\ee
\end{theorem}
The restriction of the measure to a ball $B(R)$ of $L^2$ is possible as $\Ga_\a$ leaves invariant the~$L^2(\T)$ norm for all $\a$. 
It is worthy to remark that, unlike all the other works on the subject \cite{TT10, NR-BSS11, GLV1, GLV2, ber}, we are not imposing any smallness assumption 
on $R$. This observation may be useful in the attempt of proving probabilistc global well-posedness for DNLS without imposing any smallness 
assumption on the $L^2$ norm. 
A remarkable result in this direction is \cite{JLPS} where the authors prove that DNLS is globally well-posed on the real line $\R$ in {\it weighted Sobolev spaces} using the inverse scattering method.
The well-posedness in translation invariant spaces and large $L^2$ norm remains a challenging open problem. 

The transformation of Gaussian measures have been intensively studied since long, starting from the fundamental theorem of Cameron-Martin \cite{CM} for shift maps.
The Cameron-Martin theorem was then extended in two non-overlapping directions, by Girsanov \cite{girs} (for non-anticipative maps) and  by Ramer \cite{ramer} (anticipative maps). All these results have been achieved by a functional analytic approach, exploiting the properties of the generator of the transformation, which is required to be at least of the Hilbert-Schmidt class. Nowadays these results are well established and the lay at the very basis of the development of stochastic calculus. Further developments have been achieved by means of Malliavin calculus \cite{ABC1,ABC2, mal, ust}, essentially identifying more general classes of maps allowing quasi-invariance of Gaussian measures in functional spaces. 

The problem witnessed recently a resurgence of interest mainly concerning the evolution of the Brownian motion (or related processes) along the flow of dispersive PDEs \cite{sigma, OT1, OT2, OT3, OTT, PTV, GOTW, phil, forl, deb}. A new analytic approach was introduced for flows of dispersive nonlinear equations in \cite{sigma}. The argument (inspired by the previous works \cite{TV13a,TV13b,TV14,DTV}), exploits directly the properties of the flow of the PDE under consideration. However, this technique does not provide an explicit approximation of  the density of the infinite dimensional change of coordinates induced by the flow. 

Given this framework, the DNLS gauge transformations $\Ga_\a$ represent an interesting mathematical challenge, as they are {\it anticipative maps whose generator is not Hilbert-Schmidt}. Therefore none of the classical results \cite{ramer, girs,ABC1,ABC2} applies. Nonetheless we can successfully exploit the method of \cite{sigma} and we show in this paper that we can push it to deal with regularity up to $H^{r}(\T)$ for any $r>0$, corresponding to Gaussian measures $\g_s$ with $s>\frac12$. Our result extends
(and improves, getting rid of the small $L^2$ norm restriction) the earlier analysis of \cite{NR-BSS11,GLV2} valid for integer $s\geq1$.  It will become clear from our analysis that the restriction $s>\frac12$ is most likely optimal. In particular for lower values of $s$ the support of $\gamma_s$ is no longer on classical functions and the extension of our result to such values of $s$ would certainly require some renormalisation procedure. 

A special mention is deserved by the case $s=1$, addressed in \cite{NR-BSS11} by the following nice probabilistic argument. The typical trajectories of $\g_1$ are complex Brownian bridges, for which modulus and phase are conditionally independent after a time-change. Since the gauge acts in fact as a modulus-dependent phase-shift, conditionally on the modulus Cameron-Martin theorem applies. In this way the authors were able not only to prove quasi-invariance but also the precise change of variable formula via Cameron-Martin theorem. Unfortunately this trick is very much based on the specific properties of the complex Brownian bridge and seems to be difficult to reproduce for general Gaussian measures. We stress that the small mass restriction here emerges by the so-called Novikov condition, which amounts to require uniform integrability of the density of the change of variables. To ensure this property the authors rely on the analysis of \cite{TT10} of the Gibbs measure of the derivative NLS, as the leading order terms in the exponent of the density are the same. 

Even though we do not attack the problem directly, our work gives a strong indication that the change of variable formula established in \cite{NR-BSS11} does not require any condition on the mass. Otherwise the problem of determining the precise densities given by the gauge map is still open for $s\neq1$.  Let us point out that (with the notable exception of \cite{deb}) most of the works, appeared recently on the subject in the context of dispersive PDEs, cannot specify the Radon-Nykodim derivative by means of a suitable approximation procedure. 

The  low fractional regularity brings some  new challenges. The main problem is to find a good replacement to the classical integration by parts formula (or equivalently in our case Leibnitz formula) valid for fractional derivatives. Indeed the explicit representation of the variation of the Sobolev norm for integral regularity given in \cite[Lemma 2.9]{GLV2} was obtained by a direct exploit of the Leibnitz rule and does not easily generalise to fractional regularity. Therefore we have to base our analysis on a less transparent representation in terms of Fourier coefficients, which is not evidently of similar form. Indeed the gauge map can be written as identity plus smoothing only for high frequencies, but the low frequency contribution is hard to bound. Therefore we have to isolate the low frequency term and operate on it a fractional integration by parts in order to profit by a convenient cancellation given by the imaginary part as in the DNSL integrals of motion  (the same kind of difficulty is solved in \cite{phil} with similar methods). In fact, as already remarked in \cite{GLV2}, the terms appearing from the transformation along the flow of $\Ga_\a$ of the Sobolev norms are of the same type of the DNLS energies. Therefore, even though equation (\ref{eq:gauge-eq}) is much easier than DNLS, therefore the analytical difficulties are less challenging, from the probabilistic viewpoint to study the transformed Sobolev norm via $\{\Ga_\a\}_{\a\in\R}$ or the energies of DNLS is mostly equivalent. This constitutes a crucial point of this work as, albeit the flow is very regular, the Sobolev norms are hard to bound within the support of the Gaussian measure $\g_s$. 

The paper is organised as follows. In Section \ref{sect:finite} a suitable approximation of the gauge map is presented and some first properties are stated. The most important being that this approximations behave well with the finite dimensional Lebesgue measures in the frequency space (see Section \ref{sect:J}) and that the approximating gauge flow and the true gauge flow are asymptotically close. In Section~ \ref{sect:proof} we present the main argument of the proof, which is essentially the adaptation from \cite{sigma} to our case of study, leaving most of the specifics to the subsequent sections: in Section \ref{sect:J} we study the behaviour of the Jacobian determinant; in Section \ref{sect:L2} and Section \ref{sect:prop} we show that the derivative of the Sobolev norm along the flow computed in zero is a sub-exponential random variable w.r.t. $\g_s$ restricted to a ball of $L^2$. This is the most technical part of the paper. First in Section~ \ref{sect:L2} we show the quantity of interest to converge in $L^2(\g_s)$ and then in Section \ref{sect:prop} we employ the argument of \cite{B94} to show sub-exponential behaviour. 
In both sections we need to separate small and high frequencies as explained before. The splitting differs slightly in the two sections, but is similar in spirit.

\subsection*{Notations}
Throughout we denote by $u(n)$ the $n$-th Fourier coefficient of $u:\T \to \C$. 
$E_s[\cdot]$ the expectation value w.r.t. $\g_s$.
We define the fractional derivative of order $s$ as
\begin{equation}\label{FracDer}
u^{(s)} (n) = |n|^s u(n)\,.
\end{equation}
We use the following definition for the fractional Sobolev seminorm 
$$
\| u  \|_{\dot H^s}^2 = \sum_{n \in \mathbb{Z}, n \neq 0} |n|^{2s}  |u(n)|^2 
$$
and we define the fractional Sobolev norm as
\begin{equation}\label{HomogeneousSobolevNorm}
 \| u \|_{H^s}^2  =  \| u \|_{L^2}^2 + \|u \|_{\dot{H}^{s}}^2.
\end{equation}
We have by Plancherel
\begin{equation}\label{HsInSpace}
\| u \|_{\dot{H}^s}^2 = \int | u^{(s)} |^2.
\end{equation}
We also have $\| \cdot \|_{H^0(\T)} \simeq \| \cdot \|_{L^2(\T)}$. 
Given $R>0$, we denote with $B(R)$ the ball of center zero and radius~$R$ in the $L^2(\T)$-topology.
We set $$\tilde\gamma_s(A):=E_s[1_{B(R)\cap A}].$$
Let $P_N$ be the the projection on the first $N$ Fourier modes 
\begin{equation}\label{FirstNFourier}
P_{N} \sum_{n \in \mathbb{Z}} e^{inx} c_n := \sum_{|n| \leq N} e^{inx} c_n  .
\end{equation}
We define $$\tilde\gamma_{s,N}(A):=E_s[1_{\{\|P_Nu\|_{L^2}\leq R\}\cap A}]$$ for any measurable $A$. Note that $R$ is always implicit in the definition of $\tilde\gamma_s$ and $\tilde\gamma_{s,N}$. 
For $j\in\N$ the Littlewood-Paley projector is denoted by $\D_j:=P_{2^{j}}-P_{2^{j-1}}$; we write $|n|\simeq2^{j}$ to shorten $2^{j-1}< |n|\leq 2^j$ for $j\in\N$, while for $j=0$ $|n|\simeq 1$ shortens $|n|\leq 1$.
We write $X \lesssim Y$ to denote that $X \leq C Y$ for some positive constant $C$ independent on $X,Y$.

We denote by $E_N^{\perp}$ the orthogonal complement of $E_N$ in the topology of $L^2(\T)$. 
Letting $\gamma_{s}^{\perp, N}$ the measure induced on $E_N^{\perp}$ by the map
\begin{equation}\label{Def:gammaK2}
\omega\longmapsto \sum_{|n| >N } \frac{g_n(\omega)}{(1+|n|^{2s})^{\frac{1}{2}}}\, e^{inx} \, ,
\end{equation}
the measure $\gamma_s$ factorises over $E_N \times E_N^{\perp}$ as
\begin{equation}\label{GammaPerp}
\gamma_s(du) := \frac{1}{Z_N} e^{-\frac12  \| P_N u \|_{H^s}^2 } L_N (d P_N u) \, \gamma_{s}^{\perp, N}(d P_{>N} u),
\end{equation}
where $L_N$ is the Lebesgue measure induced on $E_N$ by the isomorphism between~$\R^{2(2N+1)}$ and $E_N$ and 
$Z_N$ is a renormalisation factor.

We will use the Bernstein inequality for estimating tail probabilities in the following form.
Let $X_1,\ldots,X_N$ i.i.d. sub-exponential component of $X\in\R^N$ and $a\in\R^N$. Then there are $c,C>0$ such that
\be\label{eq_Bernstein}
P\left(|\sum_{i=1}^N a_iX_i|\geq t\right)\leq \exp\left(-c\min\left(\frac{t}{\max_i |a_i|},\frac{t^2}{\|a\|_2^2}\,\right)\right)\,.
\ee

%
%%%%%%%%%%%%%%%%%%%%%%%%%%%%%%%%%%%%%%%%%%%%%%%%%%%%%%%%%%%%%%%%%%%%%%%%%%%%%%%%%%%%%%%%%%%%%%%%%%%%%%%%%%%%%%%%%%%%%%%%%%%
%
%
\section{Approximated flow}\label{sect:finite}

Let $P_N$ be the the projection on the first $N$ Fourier modes (see \eqref{FirstNFourier}). 
Given $N \in \mathbb{N}$ we define
an approximation $\Ga_\a^N : L^2(\T) \to L^2(\T)$ of $\Ga_\a$ using the 
truncated system (compare with \eqref{eq:gauge-eq})
\be\label{eq:gaugeN-cauchy}
\frac{d}{d\a}\Ga_\a^N (u)=iP_N\big(\mathcal I[P_N \Ga_{\a}^N(u)] P_N \Ga_{\a}^N(u)\big)\,,\quad \Ga^N_0 (u) = u \,.
\ee  
By convention $P_{\infty}=\Id$. 
It is immediate to show that, for all $N \in \N$ the flow map is globally (in time) well defined, since the frequencies $>N$ evolves under the identity map and the 
frequencies~$\leq N$ evolves as the solution of an 
ODE with conserved $L^2$ norm (see Lemma \ref{L2NormGauge}). In the case~$N = \infty$ this is a consequence of the explicit representation formula \eqref{gauge-change}.
An immediate consequence of \eqref{L2Conservation} is that
$$
\| \Ga_\a^N u \|_{L^{2}} = \| u \|_{L^{2}}, \qquad \forall \a \in \R,
$$
note, however, that $|u|\neq|\Ga_\a^N u |$, which is only the case when $N = \infty$ (see again \eqref{gauge-change}).

It is also clear, looking at the definition \eqref{eq:gaugeN-cauchy}, that the 
map $\alpha \to \Ga_\a^N$ is a one parameter group of transformations, for all $N \in \mathbb{N} \cup \{ \infty \}$. 

\begin{lemma}\label{L2NormGauge}
Let $N \in \N \cup \{ \infty \}$. For all $u \in L^{2}(\T)$ we have
\begin{equation}\label{L2Conservation}
\| P_N \Ga_\a^N u \|_{L^{2}} = \| P_N u \|_{L^2} 
\end{equation}
\end{lemma}

For the proof of \eqref{L2Conservation} we refer to \cite[Section 6]{GLV2}.
Moreover, we have the following $L^2$-stability result.

\begin{lemma}\label{H1NormGauge}
Let $N \in \N \cup \{ \infty \}$. Then
\begin{equation}\label{eq:L^2ContinuityPreq}
\| \Ga_\a^N u -  \Ga_\alpha^N v \|_{L^2} \lesssim 
 e^{C|\alpha| \left( \|P_N u\|^2_{L^2} + \|P_N v\|^2_{L^2} \right) } \|  ( u - v) \|_{L^2}\,.
\end{equation}
\end{lemma}

\begin{proof}
Until the end of the proof $N \in \N \cup \{ \infty \}$.
Decomposing 
\begin{align}\nonumber
\Ga_\a^N u -  \Ga_\alpha^N v 
& =
P_N \Ga_\a^N u -  P_N \Ga_\alpha^N v
+ (\Id - P_N)  ( \Ga_\a^N u ) - (\Id - P_N) ( \Ga_\a^N v )
\\ \nonumber
& = P_N \Ga_\a^N u -  P_N \Ga_\alpha^N v
+ (\Id - P_N)  (  u -  v ),
\end{align}
where the second identity follows by the fact that $\Ga_\a^N$ is the identity on the frequencies $> N$ (remember definition \eqref{eq:gaugeN-cauchy}),
the \eqref{eq:L^2ContinuityPreq} follows by
\begin{equation}\label{eq:L^2Continuity}
\| P_N \Ga_\a^N u - P_N \Ga_\alpha^N v \|_{L^2} \lesssim  e^{C|\alpha| \left( \|P_N u\|^2_{L^2} + \|P_N v\|^2_{L^2} \right) } \| P_N ( u - v) \|_{L^2}\,.
\end{equation}
To prove \eqref{eq:L^2Continuity} we will need the inequalities 
\begin{equation}\label{Triv1}
\| \mathcal{I}[P_N \Ga^N_\a u] \|_{L^{\infty}} \lesssim \| P_N  \Ga^N_\a u \|^2_{L^{2}} = \| P_N u\|_{L^2}^2  \,,
\end{equation}
and
\begin{align}\label{Triv2Preq}
 \| \mathcal{I}[P_N \Ga^N_\a u] - \mathcal{I}[P_N \Ga^N_\a v]\|_{L^{\infty}}
& 
\lesssim \| P_N \Ga^N_\a u  + P_N \Ga^N_\a v \|_{L^2} \| P_N \Ga^N_\a u  - P_N  \Ga^N_\a v \|_{L^2} 
\\ \nonumber
& \lesssim \left( \| P_N u \|_{L^2} + \| P_N v \|_{L^2}  \right)  \| P_N \Ga^N_\a u  - P_N  \Ga^N_\a v \|_{L^2} \, .
\end{align}
These follow immediately recalling the form \eqref{DefMathcali} of $\mathcal{I}[\cdot]$ and \eqref{L2Conservation}.
Let
$$
\delta_{\alpha}^N(u,v) :=  P_N \Ga^N_\a u - P_N \Ga^N_\a v \,.
$$
Notice that $\delta^N_\a (u,v)$ solves
\begin{equation}\nn
\frac{d}{d \a} \delta^N_\a (u,v)
= 
 i  P_N \big( \mathcal{I}[ P_N \Ga^N_\a u]\delta^N_\a (u,v) + \left( \mathcal{I} [P_N \Ga^N_\a u] - \mathcal{I}[P_N \Ga^N_\a v] \right) P_N \Ga^N_\a v  \big) \, .
\end{equation}
Pairing this in $L^2$ with $\delta^N_\a (u,v)$ we get
\begin{equation}\nonumber
\frac{d}{d \a}
\| \delta^N_\a (u,v) \|_{L^2}^2 
= 2 \Re i \Big(  
 \int   \mathcal{I}[P_N \Ga^N_\a u] |\delta^N_\a (u,v)|^2  
+ 
\int \left(  \mathcal{I}[P_N \Ga^N_\a u] - \mathcal{I}[P_N \Ga^N_\a v] \right) (P_N \Ga^N_\a v) \,  \overline{\delta^N_\a (u,v)} \Big)   \, .
\end{equation}
Using the H\"older and Cauchy--Schwartz inequalities and \eqref{Triv1}, \eqref{Triv2Preq}, we arrive to
\begin{align}\nonumber
\frac{d}{d \a} 	\| \delta^N_\a (u,v) \|_{L^2}^2 
& 
\leq  \| \mathcal{I}[P_N \Ga^N_\a u] \|_{L^{\infty}} \| \delta^N_\a (u,v) \|^2_{L^2} 
\\ \nn
& 
+ 
\| \mathcal{I}[P_N \Ga^N_\a u] - \mathcal{I}[P_N \Ga^N_\a v] \|_{L^{\infty}}  \| P_N \Ga^N_\a v \|_{L^2} \| \delta^N_\a (u,v) \|_{L^2} 
\\ \nn
& 
\lesssim 
 (\| P_N u \|_{L^2}^2 +  \| P_N v \|_{L^2}^2  )  \| \delta^N_\a (u,v) \|^2_{L^2}    \, ,
\end{align}
so that \eqref{eq:L^2Continuity} follows by Gr\"onwall's lemma. 

\end{proof}

The flow $\Ga^N_\a$ approximates $\Ga^\infty_\a = \Ga_\a$ for large $N$ in the $L^2(\T)$ topology, uniformly on compact sets. This is proved in Lemma \ref{HsDecay}. 
Before we need the following statement.

\begin{lemma}\label{L^2Continuity}
The map
$$
(\alpha, u) \in \R \times L^{2}(\T) \to  \Ga_\a u \in L^2(\T)
$$
is continuous.
\end{lemma}

\begin{proof}
Decomposing 
$$
\| \Ga_\a u - \Ga_\beta v \|
\leq \| \Ga_\a u - \Ga_\alpha v \| + 
\| \Ga_\a v - \Ga_\beta v \|
$$
The statement easily follows by the estimate \eqref{eq:L^2Continuity} and
$$
\| \Ga_\a v - \Ga_\beta v \|_{L^2} \lesssim  |\alpha - \beta| \| v \|^3_{L^2}. 
$$
To prove this we assume $\beta < \alpha$ and we integrate \eqref{eq:gaugeN-cauchy} over $[\beta, \alpha]$, so that 
$$
\Ga_\a v - \Ga_\beta v = i  \int_{\beta}^{\alpha} \mathcal I[\Ga_{\a'}(v)]\Ga_{\a'}(v) d \alpha'.
$$
Taking the $L^{2}$ norm of this identity and using Minkowsky's and H\"older inequalities and \eqref{Triv1}, \eqref{L2Conservation} we arrive to 
\begin{align}\nonumber
\| \Ga_\a v - \Ga_\beta v \|_{L^{2}} 
& \leq \int_{\beta}^{\alpha} \| \big(\mathcal I[\Ga_{\a'}(v)]\Ga_{\a'}(v)\big) \|_{L^{2}} d \alpha'
\\ \nonumber
&
\leq \int_{\beta}^{\alpha} \| \mathcal I[\Ga_{\a'}(v)] \|_{L^{\infty}}  \| \Ga_{\a'}(v) \|_{L^{2}}  d \alpha' 
\leq \int_{\beta}^{\alpha} \|  v \|^3_{L^{2}}    d \alpha'  = |\alpha - \beta| \| v \|^3_{L^2},
\end{align}
as claimed.
\end{proof}

\begin{lemma}\label{HsDecay}
Let $N \in \N $ and $\bar \alpha \geq 0$.  
Let $A$ be a compact subset of $L^{2}(\T)$. Then 
\begin{equation}\label{StabL2Easy}
\lim_{N \to \infty}  \sup_{u \in  A, \, |\a| \leq \bar \a } \|  \Ga_\a u - \Ga_\a^N u \|_{L^2} = 0 \, . 
\end{equation}
\end{lemma}
%%%%
\begin{proof}
We decompose 
\begin{align}\nonumber
\Ga_\a u -  \Ga_\alpha^N u 
& =
P_N \Ga_\a u -  P_N \Ga_\alpha^N u
+ (\Id - P_N) (  \Ga_\a u  -   \Ga_\a^N u  )
\\ \nonumber
& = P_N \Ga_\a u -  P_N \Ga_\alpha^N u
+ (\Id - P_N)  ( \Ga_\a u  - u ),
\end{align}
where the second identity follows by the fact that $\Ga_\a^N$ is the identity on the frequencies $> N$ (remember definition \eqref{eq:gaugeN-cauchy}).

Since $A$ is compact, it is in particular bounded, so that $\| u \|_{L^2} \leq R$ for some $R > 1$, and by \eqref{L2Conservation} also $\| \Ga_\a u \|_{L^2} \leq R$.
Thus we have 
\begin{equation}\label{UT}
\lim_{N \to \infty}  \sup_{u \in  A, \, |\a| \leq \bar \a } \|  (\Id - P_N) \Ga_\a u \|_{L^2} + \| (\Id - P_N) u \|_{L^2} = 0 \, ,
\end{equation}
so that the \eqref{StabL2Easy} follows by
\begin{equation}\label{StabL2EasyBis}
\lim_{N \to \infty}  \sup_{u \in  A, \, |\a| \leq \bar \a } \|  P_N \Ga_\a u - P_N \Ga_\a^N u \|_{L^2} = 0 \, .
\end{equation}
To prove \eqref{StabL2EasyBis}
we will need the inequalities \eqref{Triv1}
and
\begin{equation}\label{Triv2}
\| \mathcal{I}[\Ga_\a u] - \mathcal{I}[P_N \Ga^N_\a u]\|^2_{L^{\infty}}
\lesssim R^2 \| P_N \Ga_\a u  - P_N \Ga^N_\a u \|_{L^2}^2 + R^2 \| (\Id - P_N) \Ga_\a u\|_{L^2}^2 \, ,
\end{equation}
valid for $N \in \N \cup \{ \infty \}$,
which follows by
 \begin{align*}
\| \mathcal{I}[\Ga_\a u] - \mathcal{I}[P_N \Ga^N_\a u]\|_{L^{\infty}} & 
 \lesssim \| \Ga_\a u  + P_N \Ga^N_\a u \|_{L^2} \| \Ga_\a u  - P_N \Ga^N_\a u \|_{L^2} 
\\ \nonumber
&
 \lesssim R \|  \Ga_\a u  - P_N \Ga^N_\a u \|_{L^2} \, ,
\end{align*}
which easily follows by the definition \eqref{DefMathcali} of $\mathcal{I}[\cdot]$ and \eqref{L2Conservation}.
Let
$$
\delta^N_\a u := P_N \Ga_\a u - P_N \Ga^N_\a u \,.
$$
Notice that $\delta^N_\a u$ solves
\begin{equation}\nn
\frac{d}{d \a} \delta^N_\a u 
= 
 i  P_N \big( \mathcal{I}[\Ga_\a u]\delta^N_\a u + \left( \mathcal{I} [\Ga_\a u] - \mathcal{I}[P_N \Ga^N_\a u] \right) P_N \Ga^N_\a u  \big) \, .
\end{equation}
Pairing this in $L^2$ with $\delta^N_\a u$ we get
\begin{equation*}
\frac{d}{d \a}
\| \delta^N_\a u\|_{L^2}^2 
= 
2 \Re i \Big(
 \int   \mathcal{I}[\Ga_\a u] |\delta^N_\a u|^2  
+ 
\int \left(  \mathcal{I}[\Ga_\a u] - \mathcal{I}[P_N \Ga^N_\a u] \right) P_N (\Ga^N_\a u)   \, \overline{\delta^N_\a u} \Big)   \, .
\end{equation*}
Using the H\"older and Cauchy--Schwartz inequalities and \eqref{Triv1}, \eqref{Triv2}, we arrive to
\begin{align}\label{StabGAlpha}
\frac{d}{d \a} 	\| \delta^N_\a u \|_{L^2}^2 
& 
\lesssim
  \| \mathcal{I}[\Ga_\a u] \|_{L^{\infty}} \| \delta^N_\a u \|^2_{L^2} 
+ 
\| \mathcal{I}[\Ga_\a u] - \mathcal{I}[P_N \Ga^N_\a u] \|_{L^{\infty}}  \| P_N \Ga^N_\a u \|_{L^2} \| \delta^N_\a u \|_{L^2} 
\\ \nn
& 
\lesssim 
 R^2 \| (\Id - P_N) \Ga_\a u  ) \|^2_{L^2} + R^2  \| \delta^N_\a u \|^2_{L^2}    \, .
\end{align}
Thus, using the fact that $\delta^N_\a u |_{\a = 0} = 0$, the Gr\"onwall's inequality gives
\begin{align}\label{CombinedWith}
  \| \delta^N_\a u \|^2_{L^2}  
  & \lesssim 
R^2 e^{R^2 |\bar \alpha|}
\int_{0}^{\alpha} 
 \| (\Id - P_N) \Ga_{\a'} u  ) \|^2_{L^2}
 \, d \alpha'
, 
\quad |\alpha| \leq |\bar{\alpha}| \, . 
\end{align}
Recalling \eqref{UT}, \eqref{L2Conservation} and using dominated convergence, it is clear that the right hand side of \eqref{CombinedWith} goes to 
zero as $N \to \infty$. On the other hand, using 
Lemma~\ref{L^2Continuity}, it is clear that the maps  
$$
(\alpha, u) \in [- \bar{\alpha}, \bar{\alpha}] \times A \to 
\Xi^{N}(\alpha, u) :=
\int_{0}^{\alpha} R^2 \| (\Id - P_N) \Ga_{\a'} u  ) \|^2_{L^2}
 \, d \alpha'
$$
are continuous, for all $N \in \mathbb{N}$.
Since $\Xi^{N}$ are defined on a compact set, are monotonic (w.r.t. $N$) and they vanish in the limit $N \to \infty$ (as we have just proved), 
by Dini's theorem we have that they converge to zero uniformly. Recalling \eqref{CombinedWith}, this complete the proof.
\end{proof}

The next result is a direct corollary of Lemma \ref{HsDecay}.  
\begin{corollary}\label{cor:StabCorImp}
Let $\e > 0$ and $\bar \alpha \geq 0$. For all compact subset $A$ of $L^2(\T)$, there exists~$N^*$ such that 
\begin{equation}\nonumber
\Ga_\a (A) \subset   \Ga^N_\a (A + B(\e))  \, . 
\end{equation}
for all $|\a| \leq \bar \a$ and for all $N > N^*$.
\end{corollary}
\begin{proof}
Let $u \in A$ and $u^N := \Ga^N_{-\a}  \Ga_{\a} u$. Since 
\begin{equation}\label{CylHp}
\Ga^N_{\a} u^N = \Ga_{\a} u,
\end{equation} 
it suffices to prove
\begin{equation}\label{a_N}
\| u -  u^N  \|_{L^2} \leq \varepsilon,
\end{equation}
for all sufficiently large $N$, uniformly in $u\in A$, $|\a| \leq \bar \a$. 
Indeed, if \eqref{a_N} holds, it means that for all $u \in A$ we have found 
$$u^N \in u + B(\varepsilon) \subseteq A + B(\varepsilon)$$ 
such that (see \eqref{CylHp})
$$
\Ga_{\a} u =  \Ga^N_{\a} u^N  \in  \Ga^N_{\a} ( A + B(\varepsilon) ).
$$   
To prove \eqref{a_N} we notice that, since $A$ is a compact subset of $L^2$ we have $\| u\|_{L^2} \leq C_A$ for all~$u \in A$.
Thus
\begin{align}\nonumber
\lim_N\sup_{u\in A, |\a| \leq \bar \a} \| u -  u^N  \|_{L^2} & = \lim_N \sup_{u\in A, |\a| \leq \bar \a}\| \Ga^N_{-\a}   \Ga^N_{\a} u -  \Ga^N_{-\a}    \Ga_{\a} u ) \|_{L^2}
\\ \nonumber
& \lesssim \lim_N \sup_{u\in A} e^{C\bar \a \left(   \| \Ga^N_{\a} u\|_{L^2} + \| P_N \Ga_{\a} u \|_{L^2} \right)} \| \Ga^N_{\a} u -    \Ga_{\a} u  \|_{L^2}
\\ \nonumber
& \leq  e^{2 C C_A^2 \bar \a } \lim_N \sup_{u\in A}\| \Ga^N_{\a} u -    \Ga_{\a} u  \|_{L^2} =0 \, .
\end{align}
where we used \eqref{eq:L^2ContinuityPreq} in the first inequality, \eqref{L2Conservation} in the second inequality and \eqref{StabL2Easy} to take the limit over $N$.
This completes the proof of Corollary~\ref{cor:StabCorImp}.
\end{proof}

%%%%%%%%%%%%%%%%%%%%%%%%%%%%%%%%%%%%%%%%%%%%%%%%%%%%%%%%%%%%%%%%%%%%%%%%%%%%%%%%%%%%%%%%%%%%%%%%%%%%%%%%%%%%%%%%%%%%%%%%%%

\section{Proof of Theorem \ref{th:gauge}}\label{sect:proof}

Here we give the main argument to prove Theorem \ref{th:gauge}, leaving all the (many) auxiliary statements to the next sections. We follow the strategy introduced in \cite{sigma}. 

First we define the measure
\begin{equation}\label{DefGammaS}
\tilde\g_{s,N}(A):=E_s[1_{\{A\cap\{\|P_Nu\|_{L^2}\leq R\}}]\,,\quad A\in\mathscr B(L^2(\T))\,. 
\end{equation}
Recall that~$\tilde\g_{s,N}(A)$ also depends on $R$, even though we will not track this dependence to simplify the notations.

Using the group property of $\{\Ga^N_\a\}_{\a\in\R}$ we can easily check that 
\begin{equation}\label{DerivativeAtAlpha=0}
\frac{d}{d\a}(\tilde \g_{s,N}\circ \Ga_{\a}^N)(A)\Big|_{\a=\bar\a}=\frac{d}{d\a}(\tilde \g_{s,N}\circ \Ga_{\a}^N)(\Ga_{\bar\a}^NA)\Big|_{\a=0} \, .
\end{equation}
Now we use the factorisation \eqref{GammaPerp} and Proposition 4.1 of \cite{sigma}, so that for all $E \subset \mathscr B(L^2)$ we have 
\begin{align}\nonumber
& \tilde \g_{s,N}(\Ga_\a^N (E))  = \int_{\Ga_\a^N (E)} \g_{s}(du) 1_{\left\{ \| P_N u \|_{L^2} \leq R \| \right\}} 
\\ \nonumber
& = \int_{E}  L_N (d P_N u)  \g_{s}^{\perp, N} (d P_{>N} u) 
1_{\left\{ \| P_N u \|_{L^2} \leq R \| \right\}} |\det D\Ga^N_\a (u)| \exp \left( -\frac{1}{2} \| P_N \Ga_\a^N u  \|_{\dot{H}_s}^2 - \frac12 \|P_N u\|_{L^2}^2 \right)  ,
\\ \nonumber
& =
\int_{E} \tilde \g_{s,N}(du) |\det D P_N \Ga^N_\a (u)| \exp\left(\frac12 \| P_N u \|^2_{\dot{H}^s}-\frac12\| P_N \Ga_\a^N u\|^2_{\dot{H}^s}\right) 
\end{align}
where $D P_N \Ga^N_\a (u)$ denotes the Jacobian matrix associated to $P_N \Ga_\a^N$ and in the  
second identity we used \eqref{L2Conservation}. 
%OLD \begin{equation}
%\int_{\Ga_{\a+\bar \a}^N(A)}
%\tilde\g_{s,N}(du)=\int_{\Ga_{\bar \a}^N(A)}|\det D\Ga^N_\a (u)|\exp\left(\frac12\| u\|^2_{H^s}-\frac12\|\Ga^N_\a u\|^2_{H^s}\right)\tilde\g_{s,N}(du)\,,
%\end{equation}
Using this identity with $E = \Ga_{\bar \a}^N A$, we arrive to 
$$
\tilde \g_{s,N}\circ \Ga_{\a}^N (\Ga_{\bar\a}^NA) = 
\int_{\Ga_{\bar\a}^NA} |\det D P_N \Ga^N_\a (u)| \exp\left(\frac12 \| P_N u \|^2_{\dot{H}^s}-\frac12\| P_N \Ga_\a^N u\|^2_{\dot{H}^s}\right) \tilde \g_{s,N}(du)
$$
so that, using \eqref{DetDerDa} to compute the derivative in $d\alpha$ at $\a =0$ of $|\det D P_N \Ga^N_\a (u)|$, we can rewrite \eqref{DerivativeAtAlpha=0} as
\begin{align}\nn
&
\frac{d}{d\a}(\tilde \g_{s,N}\circ \Ga_{\a}^N)(\Ga_{\bar\a}^NA)\Big|_{\a=0}
\\ \nn
&=\int_{\Ga_{\bar\a}^NA}\tilde\g_{s,N}(du) \eta(u) \dive P_N \left(\mathcal I[P_N u] P_N u)\right)+\int_{\Ga_{\bar\a}^NA}\tilde\g_{s,N}(du)\frac{d}{d\a}\|P_N \Ga_\a^Nu\|^2_{ \dot H^s}\Big|_{\a=0} \\
&=\int_{\Ga_{\bar\a}^NA}\tilde\g_{s,N}(du) \eta(u) \dive P_N \left(\mathcal I[P_N u] P_N u)\right)+\int_{\Ga_{\bar\a}^NA}\tilde\g_{s,N}(du)\frac{d}{d\a}\| \Ga_\a P_N u\|^2_{ \dot H^s}\Big|_{\a=0}\, \label{eq:main-inter1}
\end{align}
where $\eta(u) := i \frac{\det D P_N \Ga^N_\a (u)}{|\det D P_N \Ga^N_\a (u)|}$, so in particular $|\eta| =1$, and we also used
\begin{equation}\label{Time=0}
\frac{d}{d\a}\| P_N \Ga_\a^N u\|^2_{\dot H^s}\Big|_{\a=0} = 
\frac{d}{d\a}\| \Ga_\a P_N u\|^2_{\dot H^s}\Big|_{\a=0} \, . 
\end{equation}
To prove \eqref{Time=0}, bearing in mind \eqref{FracDer}, \eqref{HsInSpace}, \eqref{eq:gaugeN-cauchy}, we observe that
\begin{align}\nn
\frac{d}{d\a}\| P_N \Ga_\a^N u\|^2_{\dot{H}^s}\Big|_{\a=0}
&=2\Re\int (\overline{ P_N  \Ga_\a^N u})^{(s)} \left(\frac{d}{d\a} P_N \Ga_\a^N u\right)^{(s)}\Big|_{\a=0}\nn\\
&=2\Re\int \overline{P_N  \Ga_\a^N u}^{(s)} (i\mc I[ P_N  \Ga_\a^N u] P_N  \Ga_\a^N u)^{(s)}\Big|_{\a=0}\nn\\
&=2\Re\int \overline{P_N u}^{(s)} (i\mc I[P_N u] P_N u)^{(s)}=\nn \\
&=2\Re\int \overline{\Ga_\a P_N u}^{(s)} (i\mc I[  \Ga_\a P_N u] \Ga_\a P_N u)^{(s)}\Big|_{\a=0} \nn\\
&=2\Re\int (\overline{ \Ga_\a P_N u})^{(s)} \left(\frac{d}{d\a}  \Ga_\a P_Nu \right)^{(s)}\Big|_{\a=0}
= \frac{d}{d\a}\|\Ga_\a u\|^2_{\dot{H}^s}\big|_{\a=0}\,.
\label{recall}
\end{align}

Now, the first summand in (\ref{eq:main-inter1}) gives a vanishing contribution as $N\to\infty$. 
Indeed by Proposition~\ref{prop:E(div)} below and H\"older inequality there is $\e>0$ such that
\bea
\int_{\Ga_{\bar\a}^NA}\tilde\g_{s,N}(du)| \dive P_N \left(\mathcal I[P_N u] P_N u| \right)|\lesssim \frac{\tilde \g_{s,N}(\Ga_{\bar\a}^NA)^{1-\frac1p}p}{N^{\e}}\,.
\eea
On the second summand use again H\"older inequality: 
\be
\int_{\Ga_{\bar\a}^NA}\tilde\g_{s,N}(du)\frac{d}{d\a}\|\Ga_\a P_N u\|^2_{\dot H^s}\Big|_{\a=0}\leq \tilde\g_{s,N}(\Ga_{\bar\a}^NA)^{1-\frac1p}
\left\|\frac{d}{d\a}\| \Ga_\a P_N u\|^2_{\dot H^s}\Big|_{\a=0}\right\|_{L^p(\tilde \g_{s,N})}\,.
\ee

Hereafter we set 
\begin{equation}\label{Def:R^*}
R^*  := \max\left( R^{\frac{2}{2s-1}}, R^{2(2s-1)}\right).
\end{equation}

\begin{proposition}\label{prop:mom-sub-exp}
Let $s>\frac12$ and $R> 0$.
For all $N$ there is a $C>0$ such that 
\be
\left\|\frac{d}{d\a}\| \Ga_\a P_Nu\|^2_{\dot H^s}\Big|_{\a=0}\right\|_{L^p(\tilde \g_{s,N})} \leq C R^* p\,.
\ee
\end{proposition}
Altogether
\be
\frac{d}{d\a}(\tilde \g_{s,N}\circ \Ga_{\a}^N)(A)\Big|_{\a = \bar\a}\leq C R^* p\,\tilde\g_{s,N}(\Ga_{\bar\a}^NA)^{1-\frac1p}\, ,
\ee
which implies
\be
\frac{d}{d\a}\left(\tilde \g_{s,N}\circ \Ga_{\a}^N)(A)\right)^{\frac1p}\leq C R^*\,.  
\ee
Thus
\be\label{eq:start-again} 
(\tilde \g_{s,N}\circ \Ga_{\a}^N)(A)\leq (CR^* | \a| +\tilde \g_{s,N}(A)^{\frac1p})^p\leq ((CR^*)^p|\a|^p+\tilde \g_{s,N}(A))2^{p-1}\,.
\ee
 Let $\delta >0$ and $\tilde \g_{s}(A) \leq \delta$. 
Since $$1_{A \cap  \{ \|P_N u\|_{L^2} \leq R \} } \to 1_{A \cap \{ \| u\|_{L^2} \leq R \}}, \quad  \g_s\mbox{-a.s. as $N \to \infty$},$$ 
by dominated convergence we have
\be 
(\tilde \g_{s,N}\circ \Ga_{\a}^N)(A) \leq ((CR^*)^p|\a|^p+ 2 \delta)2^{p-1} \, ,
\ee
for all $N$ sufficiently large (the choice of $N$ only depends on $A$).
Now letting $\bar{\a} := \frac{1}{4C R^*}$ we have that for all $|\a| \leq \bar{\a}$: 
\be
(\tilde \g_{s,N}\circ \Ga_{\a}^N)(A)\leq \frac12\left(2^{-p}+\delta2^{p+1}\right)\,,\quad \forall p>1\,. 
\ee
Therefore for any $\e \in (0,1/2)$ we can take $p=-\log_2\e$ and see that there is $0<\d<\e^2$ such that
\be\label{eq:local-QI-N}
\tilde \g_{s}(A)\leq \d\quad\Rightarrow\quad(\tilde \g_{s,N}\circ \Ga_{\a}^N)(A)\leq\e, 
\qquad |\a| \leq \bar{\a}\,.
\ee
To upgrade \eqref{eq:local-QI-N} to the limiting version for $N\to\infty$ we use Corollary~\ref{cor:StabCorImp}. 

Let us take $R>0$ and any compact $A\subset B(R)$, such that $\tilde\g_{s}(A)\leq \d /2$. 
Since $A$ is compact, we can choose a small enough $\e'>0$ such that
\begin{equation}\label{CompactA}
\tilde\g_{s}(A+B (\e'))\leq  \d\,.
\end{equation}
By \eqref{eq:local-QI-N} for $|\a|\leq \bar\a$ we get
$$
\tilde\g_{s,N}(\Ga^N_\a(A+B(\e')))\leq \e\,.
$$
Corollary \ref{cor:StabCorImp} and the obvious inclusion $B(R) \subseteq \{\|P_Nu\|_2\leq R\}$ implies that, for all $N$ sufficiently large (again the 
choice of $N$ only depends on $A$):
$$
\Ga_\a(A) \cap B(R) \subset   \Ga^N_\a(A+B(\e'))  \cap \{\|P_Nu\|_2\leq R\} .
$$
Thus
\begin{align}\nonumber
(\tilde\g_s\circ\Ga_\a)(A) 
& =\g_s(\Ga_\a(A)\cap B(R)) 
\\ \nonumber
& \leq \g_s(  \Ga^N_\a(A+B(\e'))  \cap \{\|P_Nu\|_2\leq R\})
 = \tilde\g_{s,N}(\Ga^N_\a(A+B(\e'))) \leq \varepsilon.
\end{align}

In conclusion there exists $\bar{\alpha}$ such that the following holds. For all $\varepsilon \in (0, 1/2)$ we can take $0<\d<\e^2$ such that for 
any $|\alpha| < |\bar{\alpha}|$ and for any compact $A\subset B(R)$, we have
\be\label{eq:local-QI}
\tilde \g_{s}(A)\leq \d \quad \Rightarrow \quad (\tilde \g_{s}\circ \Ga_{\a})(A)\leq\e\,.
\ee

We can extend the previous relation to any $A \in \mathscr{B}(L^2(\T))\cap B(R)$ using the regularity of $\tilde\g_s$ (inherited by $\g_s$) by the general procedure explained in \cite[Lemma 8.1]{sigma}, which easily adapts here. 
This proves the local 
almost invariance of $\tilde \g_s$ under $\Ga_\a$, $|\alpha| \leq \bar{\a}$. Since $\bar{\alpha}$ only depends on $R$ and 
the restriction $u\in B(R)$ is invariant under $\Ga_\a$, 
we can globalise to $\alpha \in \R$ by the usual gluing procedure.

%%%%%%%%%%%%%%%%%%%%%%%%%%%%%%%%%%%%%%%%%%%%%%%%%%%%%%%%%%%%%%%%%%%%%%%%%%%%%%%%%%%%%%%%%%%%%%%%%%%%%%%%%%%%%%%%%%%%%%%%%%

Therefore we have shown (\ref{eq:THgauge}), where the density $\r_\a$ is in $L^1(\tilde\g_s)$. It remains to prove there exists $p_0>1$ such that the density lies in all the spaces $L^p(\tilde\g_s)$ for $p\in[1,p_0)$. First of all we start by a somewhat more quantitative version of~(\ref{eq:local-QI-N}). 
\begin{lemma}\label{lemma:quasi-quant}
Let $s > 1/2$. There exist $\a_0>0$ such that the following holds. For all $|\a|< \a_0$ and for all 
$A\in \mathscr B(L^2)\cap B(R)$ one has
\be\label{eq:quasi-quant}
(\tilde \g_{s}\circ \Ga_{\a})(A)\leq 2 \tilde  \g_{s}(A)^{1/2}\,. 
\ee
More precisely we can choose $\alpha_0 = c / R^*$, where $c>0$ is an absolute constant and $R^*$ is a function of the mass $R$ defined in \eqref{Def:R^*}.  
\end{lemma}

\begin{proof}
Let fix $A$ and let us start again by (see (\ref{eq:start-again}))
\begin{equation}\label{RHSOT}
(\tilde \g_{s,N}\circ \Ga_{\a}^N)(A)\leq (CR^* | \a| +\tilde \g_{s,N}(A)^{\frac1p})^p.
\end{equation}
We can assume $\tilde \g_s (A) >0$, otherwise \eqref{eq:quasi-quant} is consequence of \eqref{eq:local-QI}. Thus, since 
$$\tilde \g_{s, N}(A) \to \tilde \g_s (A) >0 \ \mbox{as} \ N \to \infty$$  
we have, for all sufficiently large $N$ (the choice of $N$ only depends on $A$)
$$
\tilde \g_{s, N}( A) \leq 2 \tilde \g_s (A)
$$
Thus we can bound the right hand side of \eqref{RHSOT} as 
\be
(C R^* |\a| +  (2 \tilde \g_s (A))^{\frac1p})^p 
= 2  \tilde\g_s( A)\left(1+\frac{2 C R^* |\a|}{ (2 \tilde \g_s(A))^{\frac1p}}\right)^p
= 2 \tilde\g_s(A)e^{p\log\left(1+  C R^* |\a| ( 2 \tilde \g_s (A))^{-\frac1p}\right)}\,.
\ee
Now we can pick
\be \label{eq:p(A)}
p=p(A) =  \log\frac{1}{2 \tilde\g_s(A)}\quad\mbox{ such that } \quad (2 \tilde\g_s(A))^{-\frac1p} =  e\,.
\ee
Thus
\be\label{Plug1}
(\tilde \g_{s,N}\circ \Ga_{\a}^N)(A)\leq 2 \tilde\g_s(P_N A)e^{p\log\left(1+ C R^* e\a\right)}\leq 2 \tilde\g_s(P_N A)e^{p  C R^* e\a}\,. 
\ee
Then we claim that 
\be\label{Plug2}
e^{p   C R^* e \a}\leq 
\tilde\g_s(A)^{-1/2}\,.
\ee
To have that, it must be
\be\label{ChoiceOfAlpha}
p  C R^* e\a\leq \frac12 \log\frac{1}{\tilde\g_s(A)} = \frac{p}{2}
\ee
which is true for $|\alpha| \leq \alpha_0$ with $\alpha_0 = c / R^*$ and $c$ sufficiently small.
Plugging \eqref{Plug2} into \eqref{Plug1} we arrive to
\be\label{eq:quasi-quantN}
(\tilde \g_{s,N}\circ \Ga_{\a}^N)(A)\leq 2 \tilde\g_s(A)^{1/2}, \qquad |\alpha| \leq \alpha_0\,.
\ee
Finally we upgrade (\ref{eq:quasi-quantN}) to (\ref{eq:quasi-quant}) using Corollary \ref{cor:StabCorImp} as in the non quantitative argument above. 
\end{proof}

The size of $\alpha_0$ in Lemma \eqref{lemma:quasi-quant} can be arbitrarily increased but paying an arbitrarily small factor loss on the 
exponent on the right hand side 
of \eqref{eq:quasi-quant} \cite[Remark 5.6]{OT2}. We have the following 

\begin{lemma}\label{lemma:quasi-quantGlob}
Let $s > 1/2$ and $\alpha \in \R$. There exist an absolute constant $\bar C > 1$ such that for all 
$A\in \mathscr B(L^2)\cap B(R)$ one has
\be\label{eq:quasi-quantFinal}
(\tilde \g_{s}\circ \Ga_{\a})(A) \leq 4  \tilde\g_{s}(A)^{\frac{1}{\bar C^{ R^* |\alpha|}}}\,. 
\ee
\end{lemma}

\begin{proof}
We can assume $\alpha \geq 0$.
Let $\alpha_0$ be given as in the previous lemma. We can also assume $\alpha_0 \in (0,1)$. 
Let now define $M$ as the largest integer such that $\alpha_0 M \leq \alpha$. 
We will show that for all $M \in \N \cup \{ 0 \}$ we have
\begin{equation}\label{Iterated}
(\tilde \g_{s}\circ \Ga_{\a})(A) \leq 2^{\sum_{j=0}^{M} 2^{-j}}  \tilde\g_{s}(A)^{2^{-(M + 1)}}, 
\quad \mbox{for $\alpha \in [\alpha_0 M, \alpha_0 (M + 1) ] $ } .
\end{equation}
Since $\alpha_0 M  \geq \alpha - \alpha_0$ (by definition of $M$) one has $2^{-(M +1) } \geq  2^{- \frac{1}{\alpha_0} \left( \alpha + \alpha_0\right)} > 2^{- \frac{1}{\alpha_0} \left( \alpha + 1 \right)}$ which means that
 the \eqref{eq:quasi-quantFinal} follows by \eqref{Iterated}, recalling that~$\alpha_0 = c/R^*$ for some absolute small constant~$c >0$.

It remains to prove \eqref{Iterated}. When $M=0$ (which means that $0\leq \alpha \leq \alpha_0$), the \eqref{Iterated} follows by \eqref{eq:quasi-quant}.
Let assume we have proved \eqref{Iterated} up to $M-1$. In particular we have
\begin{equation}\label{InductionAssumption}
\tilde \g_{s}(\Ga_{\a_0 M}(A)) = (\tilde \g_{s}\circ \Ga_{\a_0 M})(A) = ( \tilde \g_{s}(\Ga_{\a_0 M}(A)) )^{1/2}
\leq 2^{ \sum_{j=0}^{M-1} 2^{-j}} \tilde \g_{s}(A)^{2^{-M}}.
\end{equation}
Writing $\alpha \in [\alpha_0 M, \alpha_0 (M + 1) ]$ as $\alpha = \alpha_0 M + \alpha'$ with 
$|\alpha'| \leq \alpha_0$ and using \eqref{eq:quasi-quant}-\eqref{InductionAssumption} we have 
\begin{align*}
(\tilde \g_{s}\circ \Ga_{\a})(A) & =
(\tilde \g_{s}\circ \Ga_{\a'})(\Ga_{\a_0 M}(A))
\leq 2 ( \tilde \g_{s}(\Ga_{\a_0 M}(A)) )^{1/2}
\\ \nonumber
&
\leq 2 \left(  2^{ \sum_{j=0}^{M-1} 2^{-j}} \tilde \g_{s}(A)^{2^{-M}} \right)^{1/2} 
\leq  2^{ \sum_{j=0}^{M} 2^{-j}} \tilde \g_{s}(A)^{2^{-(M+1)}}  ,
\end{align*}
as claimed.

\end{proof}

We are now ready to prove that the density $\rho_\a$ is slightly more than just integrable, for all~$\alpha \in \R$.   

\begin{proposition}
Let $s > 1/2$, $R >0$ and $\alpha \in \R$. There exists $p_0(|\alpha|, R) > 1$ such that 
$$\rho_\a \in L^{p}(\g_s) \quad \mbox{for all} \quad p < p_0(|\alpha|, R).$$  
\end{proposition}
In fact we have $p_0(|\alpha|, R) \to 1$ as $|\alpha| \to  \infty$.

\begin{proof}
With reference to (\ref{eq:quasi-quantFinal}) we let for brevity 
\begin{equation}\label{Def:delta}
1 - \delta := \frac{1}{\bar C^{ R^* |\alpha|}}
\end{equation}
so that it becomes
\be\label{eq:quasi-quantFinalWithDelta}
(\tilde \g_{s}\circ \Ga_{\a})(A)\lesssim \tilde\g_{s}(A)^{1 - \delta}\,. 
\ee
Since $\bar C > 1$ we have $\delta \in (0,1)$ and $\delta \to 1$ as $\alpha \to  \infty$.
Let now $\l > 0$ and set 
$$
A_{\l,\a} :=\{u\,:\,\r_\a(u)>\l\}\,.
$$
Therefore using (\ref{eq:quasi-quantFinalWithDelta}) we have for $|\a|< \a_0$
\be
\tilde\g_s(A_{\l,\a})=\frac1\l\int_{A_{\l,\a}}\l \tilde\g_s(du)\leq \frac1\l\int_{A_{\l,\a}}\r_\a(u) \tilde\g_s(du)=\frac1\l(\tilde \g_{s}\circ \Ga_{\a})(A_{\l,\a})
\lesssim \frac{1}{\l}\tilde\g_{s}(A_{\l,\a})^{1-\d}\,. 
\ee
Consequently
\be\label{eq:gammaAlambda}
\tilde\g_s(A_{\l,\a})\lesssim \left(\frac{1}{\l}\right)^{1/ \d}\,.
\ee
Finally, we write
$$
\|\r_\a\|^p_p=p\int_0^\infty \l^{p-1}\tilde\g_s(A_{\l,\a})d\l<\infty\,,
$$
thanks to \eqref{eq:gammaAlambda} if $\d^{-1}-(p-1)>1$, that is $p<\d^{-1}$. The statement follows letting 
$p_0:= \d^{-1}$. Indeed, recalling the definition \eqref{Def:delta} of $\delta$ (in particular $\delta \in (0,1)$) and the fact that~$R^*$ only depends on~$R$, we have 
$p_0 = p_0(|\alpha|, R) > 1$, as claimed. 

\end{proof}

\section{The Jacobian determinant}\label{sect:J}
We denote the divergence operator $\dive$ when applied to an $n-$th 
dimensional vectorial function $H\,:\, E_N\times E_N\mapsto \C$ as
$$
\dive H (P_N u, P_N \bar u) = \sum_{|n| \leq N} \left( 
\frac{ \partial H_{n} }{ \partial u (n) } + \frac{ \partial \bar{H}_{n} }{ \partial \bar u (n) }  \right) \, .
$$
Let us recall Proposition 6.6 of \cite{GLV2}.
\begin{proposition}\label{FlowMapBij}
We have
\begin{equation}\label{DiffDive}
\det[(D P_N \Ga^N_{\a})(u)] = \exp \left( \int_0^\a d\a'     \dive i P_N  \left(\mathcal I [P_N \Ga^N_{\a'}(u) ] P_N \Ga^N_{\a'}(u) \right) \right) \, .
\end{equation}
\end{proposition}
Thus
\begin{equation}\label{DetDerDa}
\frac{d}{d\a}\det (D\Ga_\a^N)(u)\Big|_{\a=0}=i\dive P_N \left(\mathcal I[P_N u] P_N u)\right)\,. 
\end{equation}
We set for $s'>0$ and $n_0\in\N$
\be\label{eq:defL}
L_{s',n_0}=L_{s',n_0}[u]:=\sup_{n\geq n_0}\left(n^{s'}| u(n) |\right)\,.
\ee
\begin{lemma}\label{lemma:div1}
For all $s'\in[0,s)$ there is a $C_{s,s'}>0$ such that 
\bea
\g_s(L_{s',n_0}\geq t)&\leq &C_{s,s'}e^{-\frac {t^2}{2}n_0^{2(s-s')}}\label{eq:Csk}\\
\|L^2_{s',n_0}\|_{L^{p}(\g_s)}&\leq&\frac{C_{s,s'}p}{n_0^{s-s'}}\label{eq:Csk2}\,.
\eea
\end{lemma}
\begin{proof}
It suffices to prove (\ref{eq:Csk}), then (\ref{eq:Csk2}) readily follows. 

For $\theta>0$ a simple Gaussian integral gives
$$
E_s[e^{\theta n^{s'}| u(n) |}]=\exp\left(\frac{\theta^2}{2}\frac{n^{2s'}}{1+n^{2s}}\right)\,.
$$
Therefore by Markov inequality
$$
\g_s\left(n^{s'}| u(n) |\geq t\right)\leq \exp\left(-\theta t+\frac{\theta^2}{2}\frac{n^{2s'}}{1+n^{2s}}\right)
$$
for any $\theta>0$ and in particular picking $\theta=t(1+n^{2s})n^{-2s'}$ we have
$$
\g_s\left(n^{s'}| u(n) |\geq t\right)\leq \exp\left(-t^2(1+n^{2s})/2n^{2s'}\right)\,. 
$$
Thereby by union bound
$$
\g_s\left(L_{s',n_0}\geq t\right)\leq \sum_{|n|\geq n_0}\exp\left(-t^2(1+n^{2s})/2n^{2s'}\right) <\frac{C}{s-s'}e^{-\frac{t^2}{2}n_0^{2(s-s')}}\,,
$$
where $C$ is uniformly bounded for $n_0\in\N$. Then (\ref{eq:Csk}) follows.

\end{proof}

\begin{lemma}\label{lemma:div}
Let $s>\frac12$, $s'\in(\frac12,s)$, $\e\in(0,\frac12)$. The following bound holds
\be\label{eq:div1}
\left|  \dive P_N \left(\mathcal I[P_N u] P_N u)\right) \right| \lesssim \frac{\| P_N u\|^2_{L^2}}{N^{1-\e}}+L^2_{s', \lfloor N^\e \rfloor}\frac{\log 2N}{N^{\e(2s'-1)}}\,.
\ee
 \end{lemma}
 \begin{proof}
A direct computation from (\ref{DefMathcali}) yields
\be\label{eq:FTdi-IBis}
(\mathcal{I}[P_N u])(0) = 0, \qquad
(\mathcal{I}[P_N u])(m) = 
- \frac{i}{m} \sum_{|\ell|, |\ell - m| \leq N }  u (\ell) \bar u (\ell - m) \quad \mbox{if}  \quad m \neq 0\,, 
\ee
thus
\begin{equation}\label{FTOTRHS1}
i \left( \mathcal{I}[P_N u] P_N u \right)(n) =   \sum_{m \, : \, m \neq 0, |n-m| \leq N  } \frac{1}{m}
 \sum_{\ell \, : \, |\ell|, |\ell - m|,  \leq N }   u (n-m) u (\ell) \bar u (\ell - m) \, ,
\end{equation}
and
\bea
\dive i P_N \left(\mathcal I[P_N u ] P_N u)\right)&=&   2  \sum_{|n| \leq N} \quad \sum_{m \, : \,  m \neq 0, |n-m| \leq N }  \frac{1}{m} |u (n-m)|^{2}  \nn\\
&=&2 \sum_{n=1}^{N} \left( |u (-n)|^{2} - |u (n)|^{2} \right) \sum_{m=N-n+1}^{N+n} \frac{1}{m} \nn\,.
\eea
We pick any $\e\in(0,\frac12)$ and split the sum in $1\leq n\leq N^\e$ and $n>N^\e$. For the first part we
notice 
\begin{equation}\label{LOG}
\sum_{m=N-n+1}^{N+n} \frac{1}{m} \lesssim \ln \left( N + n \right) - \ln \left( N- n +1\right) = \ln \left( 1 + \frac{2n - 1}{N-n +1} \right)  
\end{equation}
so that we have
$$
\left|\sum_{n=1}^{\lfloor N^\e\rfloor} \left( |u (-n)|^{2} - |u (n)|^{2} \right)\sum_{m=N-n+1}^{N+n} \frac{1}{m} \right|\lesssim \|u\|^2_{L^2}\ln \left( 1 + \frac{2N^\e-1}{N-N^\e+1} \right)\lesssim  \frac{\|P_Nu\|^2_{L^2}}{N^{1-\e}}\,. 
$$
For $n>N^\e$ we use $|u (n)|^{2} \leq n^{-2s'} L^2_{s',\lfloor N^\e\rfloor}$ and \eqref{LOG} to estimate the modulus the second part of the sum as (remember $s' > 1/2$)
\begin{align}\nn
L^2_{s',\lfloor N^\e\rfloor} \sum_{n=\lfloor N^\e\rfloor+1}^{N} \frac{1}{n^{2s'}} \ln \left( 1 + \frac{2n-1}{N-n+1} \right)\leq \ln(2N) L^2_{s',\lfloor N^\e\rfloor} \sum_{n=\lfloor N^\e\rfloor+1}^{N} \frac{1}{n^{2s'}}\leq \frac{\ln 2N}{N^{(2s'-1)\e}} L^2_{s',\lfloor N^\e\rfloor} \, .
\end{align}
This completes the proof of (\ref{eq:div1}). 
\end{proof}

Therefore the following result is easily proven:

\begin{proposition}\label{prop:E(div)}
Let $R>0$, $\e\in(0,\frac12)$ small enough. There is $c(R)>0$ such that
\be
\left\|1_{\{\|u\|_{L^2}\leq R\}}\dive P_N \left(\mathcal I[P_N u ] P_N u)\right)\right\|_{L^p(\g_s)}\lesssim p\frac{c(R)}{N^\e}\,. 
\ee
\end{proposition}

%%%%%%%%%%%%%%%%%%%%%%%%%%%%%%%%%%%%%%%%%%%%%%%%%%%%%%%%%%%%%%%%%%%%%%%%%%%%%%%%%%%%%%%%%%%%%%%%%%%%%%%%%%%%%%%%%%%%%%%%%%%%

\section{$L^2(\g_s)$-convergence
}\label{sect:L2}

We start by a useful representation 
formula for 
\begin{equation}\label{Def:F}
F_N := F[P_N u]:=\frac{d}{d\a}\|\Ga_\a P_N u\|^2_{\dot H^s}\Big|_{\a=0}\,.
\end{equation}

\begin{lemma}\label{Lemma:decomposizione}
Let $s >0$, $N \in \mathbb{N}$ and $u\in L^2(\T)$. We have
\be
 F_N = F_N^{<}+F_N^{\geq}\,,
\ee
where
\begin{equation}\label{eq:F>}
F^{\geq}_N := F^{\geq}[P_Nu] 
:= 2\Re\Big(\sum_{\substack{|m_{1,2}|, |n_{1,2}|\leq N \\ n_1-m_1 \neq 0 \\ |n_1-m_1|\geq \min(|n_1|,|m_1|)\\ n_1+n_2=m_1+m_2}}\frac{|m_1|^{2s}}{m_1-n_1} u(n_1)u(n_2)\bar u(m_1)\bar u(m_2)\Big)
\end{equation}
\begin{align}\label{eq:F<}
& F^{<}_N := F^{<}[P_Nu]
\\ \nonumber
&:= -2\sum_{k\geq1}\frac{(s)_k}{k!}\Re\Big(\sum_{\substack{|m_{1,2}|,|n_{1,2}|\leq N \\  n_1 - m_1 \neq 0 \\ |n_1-m_1|<\min(|n_1|,|m_1|)\\ n_1+n_2=m_1+m_2}} \frac{(m_1-n_1)^{k-1}|m_1|^k}{|m_1n_1|^{k-s}}u(n_1)u(n_2)\bar u(m_1) \bar u(m_2) \Big)\,.
\end{align}
\end{lemma}

\begin{proof}

We use for $s>0$ the Taylor series converging for $|x|<1$
\be\label{eq:binomial}
(1+x)^s=\sum_{k\geq0}\frac{(s)_k}{k!}x^k\,,
\ee
where $(s)_k$ denotes the falling factorial
$$
(s)_0=1\,,\quad (s)_{k}:=\prod_{j=0}^{k-1}(s-j)\,,\quad k\geq1\,. 
$$
Let now  compute
\begin{align}\nonumber
(\Ga_{\a} P_N u ) (n) 
&  = 1_{[-N, N]}(n) u (n) +   \sum_{k \geq 1} \frac{(i\a)^{k}}{k !} ((\mathcal{I}[P_N u])^{k} u)(n) 
\\
&:=u (n) +i\a\sum_{|n_1|\leq N,n-p=n_1} u (n_1)(\mathcal{I}[P_N u])(p)
+\sum_{k \geq 2}  \frac{\alpha^k}{k !} r(k,n)\, .
\end{align}
Using an integration by parts in the definition of the Fourier coefficient $((\mathcal{I}[P_N u])^{k} u)(n)$ , we obtain
$$
|r(k,n)|\leq C\langle n\rangle^{-10}\|(\mathcal{I}[P_N u])^{k} P_N u\|_{H^{10}}\leq C \langle n\rangle^{-10}\ \|  P_N u\|_{H^{10}}^{2k+1}\,.
$$
Therefore,  for $\alpha$ small enough, we get the estimate 
$$
\Big|\| \Ga_{\a} P_N u  \|_{\dot H^s}^2-\| P_N u\|_{\dot H^s}^2-2\a\Im\Big(\sum_{\substack{|m_1,n_1|\leq N\\m_1-p=n_1}}|m_1|^{2s}\bar u (m_1) u (n_1)
\mathcal{I}[P_Nu](p)\Big)\Big|
\leq C_{\|P_N u\|_{H^{10}}}\, \alpha^2.
$$
Thus
$$
\frac{d}{d\a}\|\Ga_\a^N u\|^2_{\dot H^s}\Big|_{\a=0}= 
\frac{d}{d\a}\|\Ga_\a P_N u\|^2_{\dot H^s}\Big|_{\a=0}= 2\Im\Big(\sum_{\substack{|m_1,n_1|\leq N\\m_1-p=n_1}}|m_1|^{2s}\bar u (m_1) u (n_1)\mathcal{I}[P_N u](p)\Big)\,.
$$
Now we conveniently represent
$$
\frac{d}{d\a}\|\Ga_\a P_N u\|^2_{\dot H^s}\Big|_{\a=0}=F_N^{<}+F_N^{\geq} \,,
$$
where
\bea
F_N^{\geq}:=2\Im\Big(\sum_{\substack{|m_1,n_1|\leq N \\ |p|\geq \min(|n_1|,|m_1|) \\ m_1-p=n_1}}|m_1|^{2s}\bar u (m_1) u (n_1)\mathcal{I}[P_N u](p)\Big)\label{eq:F>0}\\
F_N^{<}:=2\Im\Big(\sum_{\substack{|m_1,n_1|\leq N,m_1-p=n_1\\|p|<\min(|n_1|,|m_1|)}}|m_1|^{s}|n_1+p|^s\bar u (m_1) u (n_1)\mathcal{I}[P_N u](p)\Big)\,.\label{eq:F<0}
\eea

Then (\ref{eq:F>}) is easily obtained from (\ref{eq:F>0}) by using
\be\label{eq:FTdi-I}
(\mathcal{I}[P_N u])(p) =
\left\{ 
\begin{array}{ll}
0 & \mbox{if $p = 0$} \\
- \frac{i}{p} \sum_{\substack{|n_2|, |m_2| \leq N \\ p=n_2-m_2}}  u (n_2) \bar u (m_2) & \mbox{if $p \neq 0$}\,.
\end{array}
\right.
\ee

Let us look at $F_N^{<}$. Using (\ref{eq:binomial}) we have
\bea
F_N^{<}&=&2\Im\Big(\sum_{\substack{|m_1,n_1|\leq N,m_1-p=n_1\\|p|<\min(|n_1|,|m_1|)}}|m_1|^{s}|n_1|^{s}\sum_{k\geq0}\frac{(s)_k}{k!}\frac{p^k}{|n_1|^k}\bar u (m_1) u (n_1)\mathcal{I}[P_N u](p)\Big)\nn\\
&=&2\Im\Big(\sum_{\substack{|m_1,n_1|\leq N,m_1-p=n_1\\|p|<\min(|n_1|,|m_1|)}}|m_1|^{s}|n_1|^{s}\sum_{k\geq1}\frac{(s)_k}{k!}\frac{p^k}{|n_1| ^k}\bar u (m_1) u (n_1)\mathcal{I}[P_N u](p)\Big)\label{eq:F<-intermedio}\,
\eea
as $\mc I[u](p)=\overline{ \mc  I[u](-p)}$ yields
\be\label{eq:canc}
\Im\Big(\sum_{\substack{|m_1,n_1|\leq N,m_1-p=n_1\\|p|<\min(|n_1|,|m_1|)}}|m_1|^{s}|n_1|^{s}\bar u (m_1) u (n_1)\mathcal{I}[P_N u](p)\Big)=0\,. 
\ee
When we plug (\ref{eq:FTdi-I}) in (\ref{eq:F<-intermedio}) we obtain (\ref{eq:F<}). 
\end{proof}

\begin{remark}
By the same argument we also have
\be 
F_N=2\Re\Big(\sum_{\substack{|m_{1,2}||n_{1,2}|\leq N\\ n_1+n_2=m_1+m_2 \\ n_1 - m_1 \neq 0}}\frac{|m_1|^{2s}}{m_1-n_1} u(n_1)u(n_2)\bar u(m_1)\bar u(m_2)\Big)\label{eq:F-2}
\, .
\ee
\end{remark}

In what follows we shall use the Wick formula for expectation values of multilinear forms of Gaussian random variables in the following form. 
Let $\ell \in \N$ and $S_{\ell}$ be the 
symmetric group on $\{1,\dots,\ell\}$, whose elements are denoted by $\s$. 
Recalling \eqref{Def:gammaK}
we have
\begin{align}\label{eq:Wick}
E_s\Big[ \prod_{j=1}^{\ell}  u(n_j)  \bar u(m_j)  \Big] 
& = 
\sum_{\sigma \in S_{\ell}}\prod_{j=1}^{\ell} \frac{\d_{m_j,n_{\sigma(j)}}}{1 + |n|^{2s}} 
\\ \nonumber
& \simeq
\sum_{\sigma \in S_{\ell}}\prod_{j=1}^{\ell} \frac{\d_{m_j,n_{\sigma(j)}}}{\meanv{n_j}^{2s}} 
 \, ,
\end{align}
where $\meanv{\cdot} = (1 + |\cdot|^2)^{1/2}$.
We convey that the labels $m_i$ (respectively $n_i$) are associated to the Fourier coefficients of $\bar u$ (respectively $u$). We say that $\sigma$ contracts the pairs of indexes $(m_j,n_{\sigma(j)})$ and we shorten for any $\Omega\subset \Z^{\ell}\times\Z^\ell$
$$
\s(\Omega):=\Omega\cap\{m_i=n_{\s(i)}\,,i=1,\ldots,\ell\}\,,\quad \s\in S_\ell\,,
$$
We also define the set $\bar \Omega$ to be obtained by $\Omega$ swapping the role of $n_i$ and $m_i$ $i=1,\ldots \ell$. 

The following elementary bound will  be useful in the proof of Lemma \ref{lemma:L2-F1}.

\begin{lemma}
Let $s > \frac{1}{2}$. Then 
\begin{equation}\label{LemmaConv1}
\sum_{q \in \mathbb{Z}} \frac{1}{\meanv{p - q} \meanv{q}^{2s}} = 
\sum_{q \in \mathbb{Z}} \frac{1}{\meanv{p - q}^{2s} \meanv{q}} \leq \frac{C}{\meanv{p}} \, ,
\end{equation}
for a constant $C$ which only depends on $s$. 
Similarly
\begin{equation}\label{LemmaConv2}
\sum_{q \in \mathbb{Z}} \frac{1}{\meanv{p - q^2} \meanv{q}^{2s}} 
 \leq \frac{C}{\meanv{p}} \, .
\end{equation}
\end{lemma}
\begin{proof}
The identity in \eqref{LemmaConv1} is just a change of variables. We just show \eqref{LemmaConv1}, namely
$$
\sum_{q \in \mathbb{Z}} \frac{1}{\meanv{p - q}^{2s} \meanv{q}} \leq \frac{C}{\meanv{p}} \, ,
$$
being the proof of \eqref{LemmaConv2}
almost identical.
First we notice that
$$
\sum_{q \in \mathbb{Z} : |q| \geq \frac{|p|}{2} }   \frac{1}{\meanv{p - q}^{2s} \meanv{q}}
\leq
\frac{2}{\meanv{p}} \sum_{q \in \mathbb{Z} }
\frac{1}{\meanv{p - q}^{2s}}   
\leq \frac{C}{\meanv{p}} \, .
$$ 
Then we notice that for $|q| < \frac{|p|}{2}$ we have by triangle inequality $|p - q| \geq  |p| - |q| > \frac{|p|}{2}$, so that
\begin{align}
& \sum_{q \in \mathbb{Z} : |q| < \frac{|p|}{2} }   \frac{1}{\meanv{p - q}^{2s} \meanv{q}}
\leq
\frac{2}{\meanv{p}} \sum_{q \in \mathbb{Z} }
\frac{1}{\meanv{p - q}^{2s - 1}} \frac{1}{\meanv{q}}  
\\ \nonumber
&
\quad \quad \quad \quad  \leq \frac{2}{\meanv{p}} 
\Big( \sum_{q \in \mathbb{Z} }
\frac{1}{\meanv{p - q}^{2s}} \Big)^{\frac{2s -1}{2s}} 
\Big( \sum_{q \in \mathbb{Z} }
\frac{1}{\meanv{q}^{2s}} \Big)^{\frac{1}{2s}} \leq \frac{C}{\meanv{p}} \, ,
\end{align}
where we used H\"older inequality with the conjugate pair $\frac{2s}{2s - 1}, 2s$
which concludes the proof.
\end{proof}

\begin{lemma}
Let $\varepsilon \in (0,1)$. Then
\begin{equation}
\label{addProof?Correct}
 \sum_{|q| \in \mathbb{Z}, q \neq p}\frac{1}{\meanv{q^2 - p^2}} \leq  \frac{C_\varepsilon}{\meanv{p}^{1 - \varepsilon}} 
\end{equation}
\end{lemma}

\begin{proof}
It is clearly sufficient to show that for all $p \in \N \setminus \{ 0 \}$ we have
$$
 \sum_{q \geq 1, q \neq p}\frac{1}{\meanv{q^2 - p^2}} \leq  \frac{C_\varepsilon}{p^{1 - \varepsilon}} 
$$
We split the sum over $q = 1, \ldots, p-1$ and $q \geq p+1$. In the first case we 
estimate
\begin{align}
 \sum_{q = 1}^{p-1} \frac{1}{\meanv{q^2 - p^2}} \lesssim
   \sum_{q = 1}^{p-1}  \frac{1}{p^2 - q^2} =    \sum_{q = 0}^{p-1} \frac{1}{2p} \left( \frac{1}{p-q} + \frac{1}{p+q} \right)
   \lesssim \frac{1}{p}   \sum_{q = 1}^{2p-1} \frac{1}{q} \lesssim \frac{\ln (2p)}{p} .
\end{align}
In the second case we estimate
\begin{align}
 \sum_{q \geq p+1} \frac{1}{\meanv{q^2 - p^2}} 
 & \lesssim
   \sum_{q \geq p+1}  \frac{1}{q^2 - p^2} =    \sum_{q \geq p + 1} \frac{1}{2q} \left( \frac{1}{q-p} + \frac{1}{q + p} \right)
\\ \nonumber
&
   \leq \frac{1}{p^{1-\varepsilon}}   \sum_{q  \geq p + 1} \frac{1}{q^{\varepsilon}}  \left( \frac{1}{q-p} + \frac{1}{q + p} \right) 
\\
&
\leq \frac{1}{p^{1-\varepsilon}}  \left( \sum_{q  \geq 1} \frac{1}{q^{1 + \varepsilon}} \right)^{\frac{\varepsilon}{1 + \varepsilon}}  
\left( \sum_{q  \geq p + 1} \left( \frac{1}{q-p} + \frac{1}{q + p} \right)^{1+ \varepsilon} \right)^{\frac{1}{1+ \varepsilon}}
\label{eq:ultima-ultimolemma}
\end{align}
where we used H\"older inequality with the conjugate pair $\frac{1 + \varepsilon}{\varepsilon}, 1 + \varepsilon$. Of course the second factor in the above formula is finite for any $\e>0$. The third factor is also finite as
$$
\left( \sum_{q  \geq p + 1} \left( \frac{1}{q-p} + \frac{1}{q + p} \right)^{1+ \varepsilon} \right)^{\frac{1}{1+ \varepsilon}}\leq C_\e\left( \sum_{q  \geq  1} \left( \frac{1}{q}\right)^{1+ \varepsilon} + \sum_{q  \geq 1+p}\left(\frac{1}{q + p} \right)^{1+ \varepsilon} \right)^{\frac{1}{1+ \varepsilon}}
$$
which is bounded for $\e>0$.
We conclude
$$
(\ref{eq:ultima-ultimolemma}) \leq  \frac{C_\varepsilon}{p^{1 - \varepsilon}}
$$
which ends the proof. 
\end{proof}

\begin{lemma}\label{lemma:L2-F1}
Let $s > 1/2$ and $M \in \mathbb{N}$. Then $F_M \to F$ in $L^2(\g_s)$ as $M \to \infty$, with 
\be\label{eq:L2-f1}
\|F_M - F\|_{L^2(\g_s)}\lesssim \frac{1}{M^{s-\frac12}}\,. 
\ee
\end{lemma}
\begin{proof}
Let $N>M$ and define for $a,b\in\N$
\begin{align}\label{eq:A-NM}
A^{a,b}_{N,M}:= 
\{
&
   |n_{a,b}|,|m_{a,b}|\leq N,\, |n_a-m_a|\geq \min(|n_a|,|m_a|), \, n_a \neq m_a\, ,
\\ \nonumber
&
 n_a+n_b=m_a+m_b\,,\,\max(|m_{a,b}|,|n_{a,b}|)>M\}\,.
\end{align}
We have
\begin{equation}\label{takingsquare}
F_N-F_M=
2\Re\Big( \sum_{A^{1,2}_{N,M}}\frac{|m_1|^{2s}}{m_1-n_1}
\bar u(m_1)\bar u(m_2)u({n_1})u({n_2})\Big)\,.
\end{equation}
Now we square
$$
|F_N(u)-F_M(u)|^2=4\sum_{A^{1,2}_{N,M}\times A^{3,4}_{N,M}}
\frac{|m_1|^{2s}}{m_1-n_1} \frac{|n_3|^{2s}}{n_3-m_3}
\prod_{j=1}^4\bar u({m_j})u({n_j})\,
$$
and take the expected value w.r.t. $\g_s$ using formula \eqref{eq:Wick} with $\ell=4$
\bea
\frac14\|F_N-F_M\|_{L^2(\gamma_s)}^2&=&\sum_{A^{1,2}_{N,M}\times A^{3,4}_{N,M}}
\frac{|m_1|^{2s}}{m_1-n_1} \frac{|n_3|^{2s}}{n_3-m_3}
E_s\Big[\prod_{j=1}^4\bar u({m_j})u({n_j})\Big]\nn\\
&=&\sum_{A^{1,2}_{N,M}\times A^{3,4}_{N,M}}
\frac{m_1^{2s}}{m_1-n_1} \frac{n_3^{2s}}{n_3-m_3} \sum_{\sigma \in S_{4}}\prod_{j=1}^{4} \frac{\d_{m_j,n_{\sigma(j)}}}{\meanv{n_j}^{2s}}\label{eq:Wick4}\,. 
\eea
Therefore
\be\label{eq:Wick-intermed1}
(\ref{eq:Wick4})=\sum_{\sigma \in S_{4}}\sum_{\s(A^{1,2}_{N,M}\times A^{3,4}_{N,M})}
\frac{|n_{\s(1)}|^{2s}}{n_{\s(1)}-n_1} \frac{|n_{3}|^{2s}}{n_3-n_{\s(3)}} \prod_{j=1}^{4} \frac{1}{\meanv{n_j}^{2s}}\,.
\ee
We note that contractions with $\s(j)=j$ for some $j \in \{1, \ldots, 4 \}$ yields $\s(A^{1,2}_{N,M}\times A^{3,4}_{N,M})=\emptyset$. Therefore the sum over $\s$ runs actually in
$$
S'_4:= \{ \s \in S_4 : \s(j) \neq j,\forall j =1,\ldots,4\})\,,
$$
which contains 9 elements. 
Moreover the contribution of $\s=(2,1,4,3)$ to \eqref{eq:Wick-intermed1} is 
$$
\sum_{ \substack{ |n_i| \leq N, i = 1, \ldots, 4 \\  n_2 \neq n_1,  n_4 \neq n_3  \\ \max(|n_1|,|n_2|), \max(|n_3|,|n_4|) > M } }\frac{|n_2|^{2s}}{n_2-n_1}\frac{|n_3|^{2s}}{n_4-n_3}=\Big(\sum_{\substack{|n_1|,|n_2|\leq N\\ n_2 \neq n_1 \\ \max(|n_1|,|n_2|)\geq M}}\frac{|n_2|^{2s}}{n_2-n_1}\Big)^2\,,
$$
which is
zero, due to its antisymmetry w.r.t. $n_1\mapsto -n_1, n_2\mapsto -n_2$.
It remains to consider the following 8 permutations:

\begin{equation}
\sigma = \left\{
\begin{array}{ll}
(a) & (3,1,4,2) \\
(b) & (4,1,2,3) \\
(c) &  (2,3,4,1) \quad \mbox{(reduces to $(b)$)} \\
(d) & (2,4,1,3) \\
(e) & (4,3,2,1) \\
(f) & (3,4,2,1) \\
(g) & (3,4,1,2) \\
(h) & (4,3,1,2) 
\end{array}
\right.
\end{equation}

$\bullet\,$Case (a), $\s = (3,1,4,2)$. Note that $n_1 + n_2 = n_{\s(1)} + n_{\s(2)}$ and $n_3 + n_4 = n_{\s(3)} + n_{\s(4)}$ reduces to $n_2 = n_3$, so we need to evaluate
\begin{align}\nonumber
& 
\Big| \sum_{ \substack{ |n_i| \leq N, i = 1, 2, 4 \\  n_2 \neq n_1, n_4 \neq n_2 \\ \max(|n_1|,|n_2|), \max(|n_2|,|n_4|) > M } }\frac{1}{(n_2-n_1)(n_2 - n_4) \meanv{n_1}^{2s} \meanv{n_4}^{2s}} \Big|
\\ \label{GeneralA}
& \lesssim 
\sum_{ \substack{ |n_i| \leq N, i = 1, 2, 4  \\ \max(|n_1|,|n_2|), \max(|n_2|,|n_4|) > M } }\frac{1}{\meanv{n_2-n_1} \meanv{n_2 - n_4} \meanv{n_1}^{2s} \meanv{n_4}^{2s}}
\end{align}
If $|n_2| > M$, taking advantage of the symmetry w.r.t. $n_{1} \leftrightarrow n_4$, 
we get
\begin{equation}
\eqref{GeneralA}\Big|_{|n_2| > M } \lesssim\sum_{M < |n_2|\leq N}\Big(\sum_{|n_4|\leq N}\frac{1}{\meanv{n_2-n_4}}\frac{1}{\meanv{n_4}^{2s}}\Big)^2   \lesssim \sum_{M\leq|n_2|\leq N}\frac{1}{\meanv{n_2}^{2}} \lesssim\frac{1}{M}
\end{equation}
where we used \eqref{LemmaConv1} in the second inequality.   
Otherwise it has to be $|n_1| > M$, so that
\begin{align}\nonumber
\eqref{GeneralA}\Big|_{|n_1| > M } 
&
\lesssim
 \sum_{\substack{ M < |n_1| \leq N \\  |n_4|\leq N } } \frac{1}{\meanv{n_1}^{2s}}\frac{1}{\meanv{n_4}^{2s}}\Big(\sum_{|n_2|\leq N}\frac{1}{\meanv{n_2-n_1}}\frac{1}{\meanv{n_2-n_4}}\Big)
\\ \nonumber
& \leq \sum_{\substack{ M < |n_1| \leq N \\  |n_4|\leq N } } \frac{1}{\meanv{n_1}^{2s}}\frac{1}{\meanv{n_4}^{2s}}\Big(\sum_{|n_2|\leq 2N}\frac{1}{\meanv{n_2 + n_1 - n_4}}\frac{1}{\meanv{n_2}}\Big)
\\ \nonumber
&
\lesssim \sum_{\substack{ M < |n_1| \leq N \\  |n_4|\leq N} } \frac{1}{\meanv{n_1}^{2s}}\frac{1}{\meanv{n_4}^{2s}}\frac{1}{\meanv{n_1-n_4}}\lesssim 
\sum_{M < |n_1|\leq N}\frac{1}{\meanv{n_1}^{2s+1}} \lesssim\frac{1}{M^{2s}} \, ,
\end{align}
where we used \eqref{LemmaConv1} in the third inequality.

$\bullet\,$Case (b), $\s = (4,1,2,3)$. 
Note that $n_1 + n_2 = n_{\s(1)} + n_{\s(2)}$ and $n_3 + n_4 = n_{\s(3)} + n_{\s(4)}$ reduce to $n_2 = n_4$, so we need to evaluate
\begin{align}\label{GeneralB}
& 
\Big| \sum_{ \substack{ |n_i| \leq N, i = 1, 2, 3 \\  n_2 \neq n_1, n_3 \\ \max(|n_1|,|n_2|), \max(|n_2|,|n_3|) > M } }\frac{1}{(n_2-n_1)(n_3 - n_2) \meanv{n_1}^{2s} \meanv{n_2}^{2s}} \Big|
\end{align}
If we restrict the sum also to $n_3 \neq -n_2$, we ca
can exploit the symmetry $n_3\mapsto -n_3$ to bound
\begin{align}
\eqref{GeneralB}\Big|_{n_3 \neq -n_2}  &
= \frac12 \Big| \sum_{ \substack{ |n_i| \leq N, i = 1, 2, 3 \\  n_2 \neq n_1, |n_3| \neq |n_2| \\ \max(|n_1|,|n_2|), \max(|n_2|,|n_3|) > M } }\frac{1}{(n_2-n_1)}\frac{1}{\meanv{n_1}^{2s}}
\frac{1}{\meanv{n_2}^{2s}}\left(\frac{1}{n_3-n_2}-\frac{1}{n_3+n_2}\right) \Big| \nn\\
&\leq \sum_{\substack{|n_i|\leq N,\,i=1,2,3\\\max(|n_1|,|n_2|)>M, \max(|n_2|,|n_3|)>M}}\frac{1}{\meanv{n_2-n_1}}
\frac{1}{\meanv{n_1}^{2s}}\frac{1}{\meanv{n_2}^{2s}}\frac{|n_2|}{\meanv{n^2_3-n^2_2}} \nonumber\,.
\end{align}
If $|n_2| > M$, 
using 
$$
\sum_{|n_3|\leq N}\frac{1}{\meanv{n^2_3-n^2_2}}
 \leq C \, ,
$$
 we can bound 
\begin{align}
(\ref{GeneralB}) \Big|_{\substack{n_3 \neq -n_2 \\ |n_2| > M}} 
&\lesssim \sum_{M < |n_2| \leq N} \frac{|n_2|}{\meanv{n_2}^{2s}}\Big(\sum_{|n_1|\leq N}\frac{1}{\meanv{n_2-n_1}\meanv{n_1}^{2s}}\Big)\Big(\sum_{|n_3|\leq N}\frac{1}{\meanv{n^2_3-n^2_2}}\Big)\nn 
\\ \nn
&\lesssim \sum_{M < |n_2| \leq N} \frac{|n_2|}{\meanv{n_2}^{2s}}\Big(\sum_{|n_1|\leq N}\frac{1}{\meanv{n_2-n_1}\meanv{n_1}^{2s}}\Big)
\\ \nonumber
&  
\lesssim
\sum_{M < |n_2| \leq N} \frac{1}{\meanv{n_2}^{2s}}\lesssim\frac{1}{M^{2s-1}}\,,\nn
\end{align}
where we used \eqref{LemmaConv1} in the third inequality.
Otherwise it must be $|n_3| > M$, so
\begin{align}
(\ref{GeneralB}) \Big|_{\substack{n_3 \neq -n_2 \\ |n_3| > M}}
&\lesssim \sum_{ \substack{ M < |n_3| \leq N \\ |n_2| \leq N} }\frac{|n_2|}{\meanv{n^2_3-n^2_2}\meanv{n_2}^{2s}} \Big(\sum_{|n_1|\leq N}\frac{1}{\meanv{n_2-n_1}\meanv{n_1}^{2s}}\Big)\nn \\
\\ \nonumber
&
\lesssim 
\sum_{\substack{ M < |n_3| \leq N \\ |n_2| \leq N}} 
\frac{1}{\meanv{n^2_3-n^2_2}\meanv{n_2}^{2s}} 
\lesssim \sum_{M < |n_3| \leq N} \frac{1}{\meanv{n_3^2}}\lesssim\frac{1}{M}\,, \nn
\end{align}
where we used \eqref{LemmaConv1} in the second inequality and \eqref{LemmaConv2} in the third one.
When we sum over $n_3 = -n_2$ we get instead

\begin{equation}\nonumber
\eqref{GeneralB}\Big|_{n_3 = -n_2}
\lesssim \sum_{\substack{M < |n_2| \leq N \\ |n_1| \leq N, n_1 \neq n_2  }}\frac{1}{\meanv{n_2-n_1}}
\frac{1}{\meanv{n_1}^{2s}}\frac{1}{\meanv{n_2}^{2s+1}}
\lesssim \sum_{M < |n_2| \leq N }\frac{1}{\meanv{n_2}^{2s+1}}
\lesssim \frac{1}{M^{2s}} \, .
\end{equation}

$\bullet\,$Case (c), $\s = (2,3,4,1)$. Note that $n_1 + n_2 = n_{\s(1)} + n_{\s(2)}$ and $n_3 + n_4 = n_{\s(3)} + n_{\s(4)}$ reduce to $n_1 = n_3$, so we need to evaluate
\begin{equation} \label{GeneralC}
 \Big| \sum_{ \substack{ |n_i| \leq N, i = 1, 2, 4 \\  n_1 \neq n_2, n_4  \\ \max(|n_1|,|n_2|), \max(|n_1|,|n_4|) > M } }\frac{1}{(n_2-n_1)(n_1 - n_4) \meanv{n_1}^{2s} \meanv{n_4}^{2s}} \Big|
\end{equation}
that, modulo rename $(n_1, n_4, n_2) = (n_2', n_1', n_3')$, is the \eqref{GeneralB}. So we proceed as in the case (b).

$\bullet\,$Case (d), $\s = (2,4,1,3)$. Note that $n_1 + n_2 = n_{\s(1)} + n_{\s(2)}$ and $n_3 + n_4 = n_{\s(3)} + n_{\s(4)}$ reduce to $n_1 = n_4$, so we need to evaluate
\begin{equation}\label{GeneralD}
 \Big| \sum_{ \substack{ |n_i| \leq N, i = 1, 2, 3 \\  n_1 \neq n_2, n_3 \\ \max(|n_1|,|n_2|), \max(|n_1|,|n_3|) > M } }\frac{1}{(n_2-n_1)(n_3 - n_1) \meanv{n_1}^{4s} }  \Big|
\end{equation}
When we restrict the sum also over $n_1 \neq -n_2, -n_3$, we can 
exploit first the map $n_3 \leftrightarrow -n_3$ and then $n_1\leftrightarrow -n_1$ to get
\begin{align}
&
 \eqref{GeneralD}\Big|_{n_1 \neq -n_2, -n_3}  = \Big| \sum_{\substack{|n_i|\leq N,\,i=1,2,3\\ |n_1| \neq |n_2|, |n_3| \\ \max(|n_1|,|n_2|), \max(|n_1|,|n_3|)>M}}\frac{1}{(n_2-n_1)(n_3-n_1)}\frac{1}{\meanv{n_1}^{4s}} \Big|
\nn\\
&=
\frac12 \Big|
\sum_{\substack{|n_i|\leq N,\,i=1,2,3\\  |n_1| \neq |n_2|, |n_3| \\  \max(|n_1|,|n_2|), \max(|n_1|,|n_3|)>M}}\frac{1}{(n_2-n_1)}\frac{1}{\meanv{n_1}^{4s}}\Big(\frac{1}{(n_3-n_1)}-\frac{1}{(n_3+n_1)}\Big)
\Big|
\nn\\
&=
\Big| 
\sum_{\substack{|n_i|\leq N,\,i=1,2,3\\ |n_1| \neq |n_2|, |n_3| \\ \max(|n_1|,|n_2|), \max(|n_1|,|n_3|)>M}}\frac{1}{(n_2-n_1)}\frac{1}{\meanv{n_1}^{4s}}\frac{n_1}{(n^2_3-n^2_1)}
\Big|
\nn\\
&=\frac12
\Big| 
\sum_{\substack{|n_i|\leq N,\,i=1,2,3\\ |n_1| \neq |n_2|, |n_3| \\  \max(|n_1|,|n_2|), \max(|n_1|,|n_3|)>M}}\Big(\frac{n_1}{(n_2-n_1)}-\frac{n_1}{(n_2+n_1)}\Big)
\frac{1}{\meanv{n_1}^{4s}}\frac{1}{(n^2_3-n^2_1)}
\Big|
\nn\\
&\lesssim
\sum_{\substack{|n_i|\leq N,\,i=1,2,3  \\  \max(|n_1|,|n_2|), \max(|n_1|,|n_3|)>M}}\frac{1}{\meanv{n^2_2-n^2_1} \meanv{n^2_3-n^2_1} }\frac{1}{\meanv{n_1}^{4s-2}}
\nonumber\,.
\end{align}
Thus
 if $|n_1| > M$ we have, for all $\varepsilon \in (0,1)$
\begin{align}
 (\ref{GeneralD}) \Big|_{\substack{n_1 \neq -n_2, -n_3 \\ |n_1| > M}} 
 & 
 \lesssim \sum_{M <  |n_1| \leq N} \frac{1}{\meanv{n_1}^{4s-2}} 
 \Big(\sum_{|n_2|\leq N, n_2 \neq n_1}\frac{1}{\meanv{n^2_2-n_1^2}}\Big)^2 
 \\ \nonumber
 &\leq  \sum_{M < |n_1| \leq N}\frac{C_{\varepsilon}}{\meanv{n_1}^{4s - \varepsilon}}\lesssim \frac{C_{\varepsilon}}{M^{4s-1 - \varepsilon}}\,,
\end{align}
where we used the symmetry $n_2 \leftrightarrow n_3$ and the inequality \eqref{addProof?Correct}. 
Otherwise it has to be $|n_2| > M$, so that
\begin{align}
(\ref{GeneralD})\Big|_{\substack{n_1 \neq -n_2, -n_3 \\ |n_2| > M}} 
& 
\lesssim 
\sum_{\substack{M < |n_2| \leq N \\ |n_1| \leq N} } \frac{1}{\meanv{n^2_2-n^2_1} \meanv{n_1}^{4s - 2}  }
\sum_{ |n_3| \leq N, n_3 \neq n_1} \frac{1}{\meanv{n^2_3-n^2_1} }
\\ \nonumber
&
\leq 
\sum_{\substack{M < |n_2| \leq N \\ |n_1| \leq N} } \frac{C_{\varepsilon}}{\meanv{n^2_2-n^2_1} \meanv{n_1}^{4s - 1 - \varepsilon } }
\leq 
\sum_{M < |n_2| \leq N } \frac{C_{\varepsilon}}{\meanv{n_2}^2}  \lesssim \frac{C_{\varepsilon}}{M} \, ,
\end{align}
where we used \eqref{addProof?Correct} in the second inequality, taking $\varepsilon > 0$ sufficiently small, and  \eqref{LemmaConv2} in the third inequality.
It remains to consider the case when the sum \eqref{GeneralD} is taken over $n_1 = -n_2$ or $n_1 = -n_3$. 
The contribution of the terms with $n_1 = -n_2 \neq -n_3$ is again handled exploiting again the map $n_3 \leftrightarrow -n_3$, so that we have, 
for all $\varepsilon \in (0,1)$
  \begin{align}\nonumber
  \eqref{GeneralD} \Big|_{n_1 = -n_2 \neq -n_3}  
  & = 
  \frac12 \Big| \sum_{\substack{|n_1|,|n_3|\leq N, \\ |n_1| \neq |n_3| \\ |n_1| > M}} \frac{1}{n_1 \meanv{n_1}^{4s}} 
  \Big( \frac{1}{(n_3 - n_1)} - \frac{1}{(n_3 + n_1)}  \Big) \Big| 
  \\ \nonumber
  & = 
  \frac12 \Big| \sum_{\substack{|n_1|,|n_3|\leq N, \\ |n_1| \neq |n_3| \\ |n_1| > M}} \frac{1}{\meanv{n_1}^{4s}} 
   \frac{1}{(n_3^2 - n_1^2)} \Big| 
   \lesssim 
 \sum_{\substack{M < |n_1| \leq N \\ |n_3|\leq N, n_3 \neq n_1}} \frac{1}{\meanv{n_1}^{4s}} 
   \frac{1}{\meanv{n_3^2 - n_1^2}} 
  \\ \nonumber
  &
  \leq
   \sum_{M < |n_1| \leq N  } \frac{C_{\varepsilon}}{\meanv{n_1}^{4s+1 - \varepsilon}} \leq \frac{C_{\varepsilon}}{M^{4s - \varepsilon}}  \, ,
\end{align}
where we used \eqref{addProof?Correct} in the last inequality.

By symmetry w.r.t. $n_2 \leftrightarrow n_3$, using the 
triangle inequality, we can reduce to consider $n_1 = -n_2$. We have  
\begin{align}
\eqref{GeneralD} \Big|_{n_1 = n_2} 
& 
\lesssim
  \sum_{  \substack{ M < |n_1| \leq N \\  |n_3| \leq N } }\frac{1}{\meanv{n_3 - n_1} \meanv{n_1}^{4s+1} }  
  \\ \nonumber
 &
\lesssim
\sum_{|n_3| \leq N} \Big( \sum_{  \substack{ M < |n_1| \leq N  } } \frac{1}{\meanv{n_1}^{\frac{3}{2}\left(4s+1\right) } } \Big)^{2/3}
\Big( \sum_{  \substack{ M < |n_1| \leq N  } } \frac{1}{\meanv{n_3 - n_1}^3  }     \Big)^{1/3} 
\end{align}
For the contribution of the terms with $n_1 = -n_2 = -n_3$, we have  
 \begin{equation}\nonumber
  \eqref{GeneralD} \Big|_{n_1 = -n_2 = -n_3}  
  \lesssim 
   \sum_{  M < |n_1| \leq N } \frac{1}{\meanv{n_1}^{4s+2}} \lesssim \frac{1}{M^{4s + 1}}  
\end{equation}

$\bullet\,$Case (e), $\s = (4,3,2,1)$.
Note that $n_1 + n_2 = n_{\s(1)} + n_{\s(2)}$ and $n_3 + n_4 = n_{\s(3)} + n_{\s(4)}$ reduce to $n_1 + n_2 = n_3 + n_4$, so we need to evaluate
\begin{align}\nonumber 
&\Big| \sum_{ \substack{ |n_i| \leq N, i = 1, \ldots, 4  \\  n_1 \neq n_4, n_2 \neq n_3  \\ 
 n_1 + n_2 = n_3 + n_4, \ \max_i(|n_i|) > M } }
\frac{1}{(n_4-n_1)(n_3 - n_2)  \meanv{n_1}^{2s} \meanv{n_2}^{2s} } \Big|
\\ 
& \leq \label{GeneralE} 
 \sum_{ \substack{ |n_i| \leq N, i = 1, 2, 3   \\  
 \max(|n_1|,|n_2|,|n_3|) > M/3 } }
\frac{1}{\meanv{n_2-n_3}^2  \meanv{n_1}^{2s} \meanv{n_2}^{2s} } \, ,
  \end{align}
 where we used that the sum was restricted over $n_4-n_1 = n_2 - n_3$.
Since 
$$
\sum_{|n_3| \leq N}\frac{1}{\meanv{n_2-n_3}^2}  \leq \sum_{|n_3|}\frac{1}{\meanv{n_3}^2}  \leq C \, ,
$$
when $|n_1| > M/3$ we can bound
\begin{equation}
\eqref{GeneralE} \Big|_{|n_1| > M/3} \leq
\sum_{ \substack{ M/3 < |n_1| \leq N   \\  
 |n_2| \leq N } }
 \frac{1}{\meanv{n_1}^{2s} \meanv{n_2}^{2s} }
 \lesssim \sum_{ M/3 < |n_1| \leq N  }
\frac{1}{ \meanv{n_1}^{2s} } \lesssim \frac{1}{M^{2s - 1}}
\end{equation}
and when $|n_2| > M/3$ we can bound
\begin{equation}
\eqref{GeneralE} \Big|_{|n_2| > M/3} \leq
\sum_{ \substack{ M/3 < |n_2| \leq N   \\  
 |n_1| \leq N } }
 \frac{1}{\meanv{n_1}^{2s} \meanv{n_2}^{2s} }
 \lesssim \sum_{ M/3 < |n_2| \leq N  }
\frac{1}{  \meanv{n_2}^{2s} } \lesssim \frac{1}{M^{2s - 1}} \, .
\end{equation} 
If $|n_3| > M/3$ we use \eqref{LemmaConv1} to bound
\begin{align}\nonumber
\eqref{GeneralE} \Big|_{|n_3| > M/3} 
& 
\leq
\sum_{ \substack{ M/3 < |n_3| \leq N   \\  
 |n_1| \leq N } } \frac{1}{ \meanv{n_1}^{2s}}
 \sum_{|n_2| \leq N}
\frac{1}{\meanv{n_2-n_3}^2  \meanv{n_2}^{2s} } 
\\ \nonumber
&
\lesssim
\sum_{ \substack{ M/3 < |n_3| \leq N   \\  
 |n_1| \leq N } } \frac{1}{ \meanv{n_1}^{2s} \meanv{n_3}^{2s}}
\lesssim 
\sum_{ M/3 < |n_3| \leq N   } 
 \frac{1}{ \meanv{n_3}^{2s}}
  \lesssim \frac{1}{M^{2s - 1}} \, .
\end{align}

To deal with the cases $(f), (g), (h)$ we will use the decomposition from Lemma \ref{Lemma:decomposizione}
and the following elementary fact 
\begin{equation}\label{eq:important-last-L21}
|n_a-n_b|\geq \min(|n_a|,|n_b|) \Rightarrow
|n_a-n_b| \geq \frac12 \max \left( |n_a|, |n_b| \right) \,. 
\end{equation}
By symmetry w.r.t. $a \leftrightarrow b$, this is a consequence of
$$
|n_a-n_b|\geq \min(|n_a|,|n_b|) \Rightarrow
|n_a-n_b| \geq \frac{|n_a|}{2} \, .
$$
This is immediate if $|n_b| \geq \frac{|n_a|}{2}$, since then $\min(|n_a|,|n_b|) \geq \frac{|n_a|}{2}$. 
Otherwise one has $|n_b| < \frac{|n_a|}{2}$ and  by triangle inequality  
$$
|n_a - n_b| \geq |n_a| - |n_b| > |n_a| - \frac{|n_a|}{2} = \frac{|n_a|}{2} \, .
$$
$\bullet\,$Case (f), $\s = (3,4,2,1)$.
We decompose
\begin{equation}\label{GeneralF}
|F_N-F_M|^2\Big|_{\sigma = (3,4,2,1)} \leq 2 |F^{<}_N-F^{<}_M|^2\Big|_{\sigma = (3,4,2,1)}
+ 
2|F^{\geq}_N-F^{\geq}_M|^2\Big|_{\sigma = (3,4,2,1)}
\end{equation}
and we will bound these term separately. We will write $\eqref{GeneralF}^{<}$ and $\eqref{GeneralF}^{\geq}$
to denote the first and second term of the sum.  
Noting that $n_1 + n_2 = n_{\s(1)} + n_{\s(2)}$ and $n_3 + n_4 = n_{\s(3)} + n_{\s(4)}$ reduce to $n_1 + n_2 = n_3 + n_4$, we get
\begin{equation}\nonumber
\eqref{GeneralF}^{\geq} = 2 \Big| \sum_{ \substack{ |n_i| \leq N, i = 1, \ldots, 4   \\ 0 \neq |n_1-n_3|\geq \min(|n_1|,|n_3|) \\ 0 \neq |n_2-n_3|\geq \min(|n_2|,|n_3|)  \\ 
n_1 + n_2 = n_3 + n_4, \ \max_i(|n_i|) > M } }\frac{|n_3|^{2s}}{(n_3-n_1) (n_3 - n_2) \meanv{n_1}^{2s} \meanv{n_2}^{2s} \meanv{n_4}^{2s}} \Big|
\end{equation}
Since we are summing over
$$
|n_1-n_3|\geq \min(|n_1|,|n_3|) , \quad    |n_2-n_3|\geq \min(|n_2|,|n_3|) \, ,
$$
by \eqref{eq:important-last-L21} we can use 
$$
\frac{1}{|n_3-n_1|}, \frac{1}{|n_3-n_2|}  \lesssim \frac{1}{\meanv{n_3}} \, .
$$
Thus
$$
\eqref{GeneralF}^{\geq} \lesssim \sum_{\substack{|n_i|\leq N,\,i=1, 2, 4\\\max(|n_1|,|n_2|,|n_4| ) > M/3}}\frac{1}{\meanv{n_1}^{2s}\meanv{n_2}^{2s}\meanv{n_4}^{2s}}\lesssim \frac{1}{M^{2s-1}} \, .
$$
To handle $\eqref{GeneralF}^{<}$ we introduce 
  \begin{align*}
A^{a,b}_{N,M}
:=\{ & |n_{i}|,|m_{i}|\leq N,\, 0<|n_a-m_a|  < \min(|n_a|,|m_a|),
\\ \nonumber
& n_a+n_b=m_a+m_b, \, \max(|m_a|,|m_b|,|n_a|,|n_b|)>M\}\,,
\end{align*}
so that Lemma~\ref{Lemma:decomposizione} gives
\begin{align}\nonumber
 |F^{<}_N-F^{<}_M| =
2\Big|\Re\Big(\sum_{k\geq1}\frac{(s)_k}{k!}\sum_{A_{N,M}^{1,2}} \frac{(m_1-n_1)^{k-1}}{|n_1|^k} |m_1|^s |n_1|^s u(n_1)u(n_2)\bar u(m_1) \bar u(m_2) \Big)\Big|\,.
\end{align}
and squaring
\begin{align*}
& |F^{<}_N(u)-F^{<}_M(u)|^2
\\
&
=4 \sum_{A^{1,2}_{N,M}\times A^{3,4}_{N,M}}
\Big[\sum_{k,h\geq1}\frac{(s)_k(s)_h}{k!h!} \frac{(m_1-n_1)^{k-1}(n_3-m_3)^{h-1}}{|n_1|^k|m_3|^h} |m_1|^s |n_1|^s |n_3|^s |m_3|^s 
\Big]
\prod_{j=1}^4\bar u({m_j})u({n_j}) \, .
\end{align*}
We now taking the expected value w.r.t. $\g_s$, using again formula \eqref{eq:Wick} with $\ell=4$, and we restrict the sum to the contribution of 
a single permuation $\sigma$, so that we get 
\begin{align}\nonumber
&\|F^{<}_N-F^{<}_M\|_{L^2(\gamma_s)}^2\Big|_{\sigma} 
 \\ 
&
= 4  \sum_{A^{1,2}_{N,M}\times A^{3,4}_{N,M}}
\Big[\sum_{k,h\geq1}\frac{(s)_k(s)_h}{k!h!} 
\frac{(m_1-n_1)^{k-1}(n_3-m_3)^{h-1}}{|n_1|^k|m_3|^h} |m_1|^s |n_1|^s |n_3|^s |m_3|^s \Big] \prod_{j=1}^{4} \frac{\d_{m_j,n_{\sigma(j)}}}{\meanv{n_j}^{2s}} \nn \\
&= 4\sum_{\s(A^{1,2}_{N,M}\times A^{3,4}_{N,M})} \Big[\sum_{k,h\geq1}\frac{(s)_k(s)_h}{k!h!}
\frac{(n_{\s(1)}-n_1)^{k-1}(n_3-n_{\s(3)})^{h-1}}{|n_1|^k|n_{\s(3)}|^h} |n_{\s(1)} |^s |n_1|^s  |n_3|^s |n_{\s(3)}|^s
\Big]\prod_{j=1}^{4} \frac{1}{\meanv{n_j}^{2s}}
\nn \\
&\lesssim \sum_{\s(A^{1,2}_{N,M}\times A^{3,4}_{N,M})} 
\frac{|n_1|^{s-1} |n_{\s(3)}|^{s-1} |n_{\s(1)}|^s |n_3|^s}{ \prod_{j=1}^{4} \meanv{n_j}^{2s} } \, ,
\label{eq:Wick4-2New}
\end{align}
So that 
$$
\eqref{GeneralF}^{<} \lesssim \sum_{\s(A^{1,2}_{N,M}\times A^{3,4}_{N,M})} 
\frac{|n_1|^{s-1} |n_{\s(3)}|^{s-1} |n_{\s(1)}|^s |n_3|^s}{ \prod_{j=1}^{4} \meanv{n_j}^{2s} } 
 \Big|_{\sigma = (3,4,2,1)}
$$
which implies
\begin{align*}
& \eqref{GeneralF}^{<}  \lesssim
\sum_{\substack{|n_i|\leq N, i=1,\ldots 4\\n_1+n_2=n_3+n_4\\ 0<|n_2- n_{3}|  < \min(|n_2|,|n_3|) \\ 0<|n_1- n_{3}|  < \min(|n_1|,|n_3|) \\ \max_i(|n_i|)> M}}
\frac{1}{\meanv{n_1}^{s+1} \meanv{n_2}^{s+1} \meanv{n_4}^{2s}}
\\
&
\lesssim
\sum_{\substack{|n_1|,|n_2|,|n_4| \leq N \\ \max(|n_1|,|n_2|,|n_4|)\geq M/3}}
\frac{1}{\meanv{n_1}^{s+1} \meanv{n_2}^{s+1} \meanv{n_4}^{2s}} \lesssim \frac{1}{M^{2s -1}} \, .
\end{align*}

$\bullet\,$Case (g), $\s = (3,4,1,2)$.
Proceeding as in the case (f), we
decompose
\begin{equation}\label{GeneralG}
|F_N-F_M|^2\Big|_{\sigma = (3,4,2,1)} \leq 2 |F^{<}_N-F^{<}_M|^2\Big|_{\sigma = (3,4,2,1)}
+ 
2|F^{\geq}_N-F^{\geq}_M|^2\Big|_{\sigma = (3,4,2,1)}
\end{equation}
and we will bound these term separately. We will write $\eqref{GeneralG}^{<}$ and $\eqref{GeneralG}^{\geq}$
to denote the first and second term of the sum.  
Note that $n_1 + n_2 = n_{\s(1)} + n_{\s(2)}$ and $n_3 + n_4 = n_{\s(3)} + n_{\s(4)}$ reduce to $n_1 + n_2 = n_3 + n_4$, so 
\begin{equation}
\eqref{GeneralG}^{\geq} = 2 \Big| \sum_{ \substack{ |n_i| \leq N, i = 1, \ldots, 4   \\  0 \neq |n_1-n_3| \geq \min(|n_1|,|n_3|) \\  n_1 + n_2 = n_3 + n_4, \ \max_i(|n_i|) > M } }
\frac{|n_3|^{2s}}{(n_3-n_1)^2  \meanv{n_1}^{2s} \meanv{n_2}^{2s} \meanv{n_4}^{2s}} \Big|
\end{equation}
by \eqref{eq:important-last-L21} we can use 
$$
\frac{1}{|n_3-n_1|} \lesssim \frac{1}{\meanv{n_3}} \, ,
$$
so that
$$
\eqref{GeneralG}^{\geq} \lesssim \sum_{\substack{|n_i|\leq N,\,i=1, 2, 4\\\max(|n_1|,|n_2|,|n_4|) > M/3}}\frac{1}{\meanv{n_1}^{2s}\meanv{n_2}^{2s}\meanv{n_4}^{2s}}\lesssim \frac{1}{M^{2s-1}} \, .
$$ 
To handle $\eqref{GeneralG}^{<}$ we use again (see case (f))
\begin{equation}
\eqref{GeneralG}^{<} 
\lesssim \sum_{\s(A^{1,2}_{N,M}\times A^{3,4}_{N,M})} 
\frac{|n_1|^{s-1} |n_{\s(3)}|^{s-1} |n_{\s(1)}|^s |n_3|^s}{ \prod_{j=1}^{4} \meanv{n_j}^{2s} } 
 \Big|_{\sigma = (3,4,1,2)} \, ,
 \end{equation}
which implies
\begin{align*}
& \eqref{GeneralF}^{<}  \lesssim 
\sum_{\substack{|n_i|\leq N, i=1,\ldots 4\\n_1+n_2=n_3+n_4\\  0< |n_1- n_{3}|  < \min(|n_1|,|n_3|) \\ \max_i(|n_i|)> M}}
\frac{1}{\meanv{n_1}^{2} \meanv{n_2}^{2s} \meanv{n_4}^{2s}}
\\
&
\lesssim
\sum_{\substack{|n_1|,|n_2|,|n_4| \leq N \\ \max(|n_1|,|n_2|,|n_4|)\geq M/3}}
\frac{1}{\meanv{n_1}^{2} \meanv{n_2}^{2s} \meanv{n_4}^{2s}} \lesssim \frac{1}{M^{2s -1}} \, .
\end{align*}

$\bullet\,$Case (h), $\s = (4,3,1,2)$.
Proceeding as in the case (f), we
decompose
\begin{equation}\label{GeneralH}
|F_N-F_M|^2\Big|_{\sigma = (4,3,1,2)} \leq 2 |F^{<}_N-F^{<}_M|^2\Big|_{\sigma = (4,3,1,2)}
+ 
2|F^{\geq}_N-F^{\geq}_M|^2\Big|_{\sigma = (4,3,1,2)}
\end{equation}
and we will bound these term separately. We will write $\eqref{GeneralH}^{<}$ and $\eqref{GeneralH}^{\geq}$
to denote the first and second term of the sum.  

Note that $n_1 + n_2 = n_{\s(1)} + n_{\s(2)}$ and $n_3 + n_4 = n_{\s(3)} + n_{\s(4)}$ reduce to $n_1 + n_2 = n_3 + n_4$, so we need to evaluate
\begin{align}
\eqref{GeneralH}^{\geq} \lesssim\,\,\,\,
& \Big| \sum_{ \substack{ |n_i| \leq N, i = 1, \ldots, 4   \\  0 \neq |n_1-n_4|\geq \min(|n_1|,|n_4|) \\  0 \neq |n_1-n_3|\geq \min(|n_1|,|n_3|) \\ 
 n_1 + n_2 = n_3 + n_4, \ \max_i(|n_i|) > M } }
\frac{1}{(n_4-n_1)(n_3 - n_1)  \meanv{n_1}^{2s} \meanv{n_2}^{2s} } \Big|
\\ \nonumber
&
\lesssim
 \sum_{ \substack{ |n_i| \leq N, i = 1, \ldots, 4  \\ 
\max(|n_1|,|n_2|,|n_4|) > M/3 } }
\frac{1}{\meanv{n_4} \meanv{n_2 - n_4}  \meanv{n_1}^{2s} \meanv{n_2}^{2s} } 
\end{align}
where we used \eqref{eq:important-last-L21}, so that
 $$
 \frac{1}{|n_4 - n_1|} \lesssim \frac{1}{\meanv{n_4}} 
 $$
 and the fact that the sum were restricted to $n_3 - n_1 = n_2 - n_4$
  If $|n_1| > M/3$  
  we can estimate
  \begin{align}\nonumber
\eqref{GeneralH}^{\geq}  \Big|_{|n_1| >M/3 }  
  & \lesssim
 \sum_{ \substack{ M/3 < |n_1| \leq N, \\ |n_4| \leq N}   } \frac{1}{\meanv{n_1}^{2s} \meanv{n_4} }  
 \sum_{ |n_2| \leq N} \frac{1}{\meanv{n_2}^{2s} \meanv{n_2 - n_4}  }
 \\ \nonumber
 &
 \lesssim
 \sum_{ \substack{ M/3 < |n_1| \leq N, \\ |n_4| \leq N }   } \frac{1}{\meanv{n_1}^{2s} \meanv{n_4}^{2} }  
 \lesssim \sum_{  M/3 < |n_1| \leq N    } \frac{1}{\meanv{n_1}^{2s}  }  \lesssim \frac{1}{M^{2s - 1}} \, ,
 \end{align}
 where in the second inequality we used \eqref{LemmaConv1}. If $|n_4| > M/3$ we estimate similarly
  \begin{align}\nonumber
  \eqref{GeneralH}^{\geq} \Big|_{|n_4| >M/3 }  
 & \lesssim
 \sum_{ \substack{ M/3 < |n_4| \leq N, \\ |n_1| \leq N}   } \frac{1}{\meanv{n_1}^{2s} \meanv{n_4} }  
 \sum_{ |n_2| \leq N} \frac{1}{\meanv{n_2}^{2s} \meanv{n_2 - n_4}  }
 \\ \nonumber
 &
 \lesssim
 \sum_{ \substack{ M/3 < |n_4| \leq N, \\ |n_1| \leq N }   } \frac{1}{\meanv{n_1}^{2s} \meanv{n_4}^{2} }  
 \lesssim \sum_{  M/3 < |n_4| \leq N    } \frac{1}{\meanv{n_4}^{2}  }  \lesssim \frac{1}{M} \, .
 \end{align}
 If $|n_2| > M/3$ we use that by Cauchy--Schwartz
 \begin{align}\nonumber
 \sum_{ |n_4| \leq N} \frac{1}{\meanv{n_4} \meanv{n_2 - n_4}  } 
 & \leq 
 \left( \sum_{ |n_4| \leq N} \frac{1}{\meanv{n_4}^2  }  \right)^{1/2} 
 \left( \sum_{ |n_4| \leq N} \frac{1}{ \meanv{n_2 - n_4}^2  } \right)^{1/2}
 \\ \nonumber
 &
 \leq \left(  \sum_{ |n_4|} \frac{1}{\meanv{n_4}^2} \right)^2 \leq C \, ,
 \end{align}
 so that
  \begin{align}
  \eqref{GeneralH}^{\geq} \Big|_{|n_2| >M/3 }  
 & \lesssim
 \sum_{ \substack{ M/3 < |n_2| \leq N, \\ |n_1| \leq N}   } \frac{1}{\meanv{n_1}^{2s} \meanv{n_2}^{2s}   }  
 \sum_{ |n_4| \leq N} \frac{1}{\meanv{n_4} \meanv{n_2 - n_4}  }
 \\ \nonumber
 &
\lesssim
 \sum_{ \substack{ M/3 < |n_2| \leq N, \\ |n_1| \leq N}   } \frac{1}{\meanv{n_1}^{2s} \meanv{n_2}^{2s}  }  
\lesssim  \sum_{ \substack{ M/3 < |n_2| \leq N, \\ |n_1| \leq N}   } \frac{1}{ \meanv{n_2}^{2s}  }  
\lesssim \frac{1}{M^{2s-1}}
 \, ,
 \end{align}
To handle $\eqref{GeneralG}^{<}$ we use again (see case (f))
\begin{equation}
\eqref{GeneralH}^{<} 
\lesssim \sum_{\s(A^{1,2}_{N,M}\times A^{3,4}_{N,M})} 
\frac{|n_1|^{s-1} |n_{\s(3)}|^{s-1} |n_{\s(1)}|^s |n_3|^s}{ \prod_{j=1}^{4} \meanv{n_j}^{2s} } 
 \Big|_{\sigma = (4,3,1,2)} \, ,
 \end{equation}
which implies
\begin{equation*}
 \eqref{GeneralH}^{<}  \lesssim 
\sum_{\substack{|n_i|\leq N, i=1,\ldots 4\\n_1+n_2=n_3+n_4\\  0< |n_1- n_{3}|  < \min(|n_1|,|n_3|) \\ 0< |n_1 - n_{4}|  < \min(|n_1|,|n_4|) \\ \max_i(|n_i|)> M}}
\frac{1}{\meanv{n_1}^{2} \meanv{n_2}^{2s} \meanv{n_3}^{s} \meanv{n_4}^{s}}
\end{equation*}
To handle this we use the following elementary fact  
\begin{equation}\label{Obv2}
|n_a - n_{b}|  < \min(|n_a|,|n_b|) \Rightarrow |n_a| \simeq |n_b|, \quad
\end{equation}
By symmetry w.r.t. $n_a \leftrightarrow n_b$ it suffices to 
show 
$$
|n_a| \leq 2 |n_b| \, ,
$$
which follows by triangle inequality
$$
|n_a| \leq |n_b| + |n_a - n_b| \leq 2 |n_b| \, ,
$$
where we used the assumption in \eqref{Obv2} in the second inequality.
Using \eqref{Obv2} we see that the sum 
above is in fact restricted to 
$$
|n_1| \simeq |n_3| \simeq |n_4| \, ,
$$
which with the restrictions
$$
n_1 + n_2 = n_3 + n_4, \quad 
\max_i(|n_i|) > M \, ,
$$
also forces
$$
|n_1| \simeq |n_3| \simeq |n_4| \gtrsim M \, .
$$
Thus we can estimate
\begin{align*}
& \eqref{GeneralH}^{<}  \lesssim 
\frac{1}{M^s}\sum_{\substack{|n_1|,|n_2|,|n_3| \leq N,  \\ |n_1|,|n_3| > M}}
\frac{1}{\meanv{n_1}^{1 + \frac{s}{2}} \meanv{n_2}^{2s} \meanv{n_3}^{1+ \frac{s}{2}} }
\lesssim \frac{1}{M^{\frac{3}{2} s}} \, ,
\end{align*}
that concludes the proof.
\end{proof}

%%%%%%%%%%%%%%%%%%%%%%%%%%%%%%%%%%%%%%%%%%%%%%%%%%%%%%%%%%%%%%%%%%%%%%%%%%%%%%%%%%%%%%%%%%%%%%%%%%%%%%%%%%%%%%%%%%%%%%%%%%%%

\section{Proof of Proposition \ref{prop:mom-sub-exp}}\label{sect:prop}

Let us recall once more
$$
F_N:=\frac{d}{d\a}\|\Ga_\a P_N u\|_{\dot H^s}\Big|_{\a=0}\,. 
$$
By Lemma \ref{lemma:L2-F1} there is $C>0$ such that for any $M \geq N \in\N$
\begin{equation}\label{Hyp}
\left\|F_N-F_M\right\|_{L^2(\g_s)}\leq \frac{C}{N^{s-\frac12}}\,.
\end{equation}
Recalling that $F_N$ can be written as in \eqref{eq:F-2}, we immediately see that \eqref{Hyp} implies by hypercontractivity that for 
all $p>2$ there is $C>0$ (possibly different from above) for which
\be\label{eq:hyper}
\left\|F_N-F_M\right\|_{L^p(\g_s)}\leq \frac{Cp^2}{N^{s-\frac12}}\,.
\ee

\ni Then we can immediately establish the following concentration inequality for $F_N$ around its limit.

\begin{proposition}\label{prop:exp0}
Let $s>1/2$ and $N\in\N$. There are $C,c>0$ such that
\be\label{eq:conc1}
\g_s\left(|F_N-F|\geq t\right)\leq Ce^{-c\sqrt t N^{\frac{2s-1}{4}}}\,. 
\ee
\end{proposition}

\begin{proof}
Having bounded all the moments as in (\ref{eq:hyper}) we can bound also the fractional exponential moment
\be\label{eq:exp.fract}
E_s\left[\exp\left(c\sqrt{|F_N-F_M|}N^{\frac {2s-1}{4}}\right)\right]<\infty\,,
\ee
for a suitable constant $c>0$ (see e.g. \cite[Proposition 4.5]{Tz10}). From (\ref{eq:exp.fract}) we obtain (\ref{eq:conc1}) in the standard way using Markov inequality. 
\end{proof}

These bounds (notably independent on $R$) are however not optimal and we need to improve on them. We shall show that $\{F_N\}_{N\in\N}$ are in fact sub-exponential random variables uniformly in $N$, whereby Proposition \ref{prop:mom-sub-exp} will follow as a simple corollary. We split the bulk and the tail of the distribution of $F$ as follows:

\begin{proposition}\label{prop:sub-exp2}
Let $s > 1/2 $ and $R^* := \max\left( R^{\frac{2}{2s-1}}, R^{2(2s-1)}\right)$. Then 
There exist $c,C>0$ independent on $N$ for which
\be\label{eq:sub-exp2}
\tilde\g_{s,N}(|F_N|\geq t)<C\exp\left(-\frac{ct}{R^* }\right)\,,\quad t\leq \sqrt{N^{2s-1}} \,.
\ee
\end{proposition}

\begin{proposition}\label{prop:sub-exp1}
Let $s > 1/2 $ and $R^* := \max\left( R^{\frac{2}{2s-1}}, R^{2(2s-1)}\right)$. There exist $c,C>0$ independent on $N$ for which
\be\label{eq:sub-exp1}
\tilde\g_{s,N}(|F_N|\geq t)<C\exp\left(-\frac{ct}{R^* }\right)\,,\quad t\geq \sqrt{N^{2s-1}} \,.
\ee
\end{proposition}

To bound the tail of the distribution of $F$ we need some work and therefore that will be addressed first.  
Recall that for $j\in\N$ the Littlewood-Paley projector is denoted by $\D_j:=P_{2^{j}}-P_{2^{j-1}}$; we write $|n|\simeq2^{j}$ to shorten $2^{j-1}\leq |n|\leq 2^j$ for $j\in\N$, while for $j=0$ $|n|\simeq 1$ shortens $|n|\leq 1$.

Let us now shorten
\bea
X_{j,N}&:=& 2^{j(s-\frac12)}\|\D_jP_Nu\|_{L^2}\,, \quad X_N := \sum_{j\geq0}X_{j,N}  \,\label{eq:defX} \\
Y_{j,N}&:=&\sum_{|n|\simeq 2^j}|u(n)|\,,\quad Y_N:= \sum_{j\geq0}Y_{j,N}\,.\label{eq:defY}
\eea
Then we obtain the following crucial lemma, whose proof can be found at the end of this section.
\begin{lemma}\label{lemma:L-Pbound}
We have
\be\label{eq:L-Pbound}
|F_N|\lesssim X_N^2Y_N^2\,.
\ee
\end{lemma}

Lemma \ref{lemma:L-Pbound} allows us to bound
\be\label{eq:union-bound-XY}
\tilde\g_{s,N}(|F_N|\geq t)\leq \tilde\g_{s,N}(X_NY_N\geq \sqrt t)\leq \tilde\g_{s,N}(X_N\geq t^{\frac12-\frac{1}{4s}})+\tilde\g_{s,N}(Y_N\geq t^{\frac{1}{4s}})\,,
\ee
so that we can treat the two contributes separately.

\begin{lemma}\label{lemma:subexpX}
There are $C,c>0$ and $l>1$ such that
\be\label{eq:sub-expX}
\tilde\g_{s,N}(X_N\geq t^{\frac12-\frac{1}{4s}})\leq \exp\left(-\frac{ct}{R^{\frac{2}{2s-1}}}\right)\,,\quad\forall t\geq (\log_2N)^{\frac{4s}{2s-1} l}\,. 
\ee
\end{lemma}
\begin{proof}

Let $$j_t:=\left\lfloor\log_2\frac{t^{\frac{1}{2s}}}{R^\frac{2}{2s-1}}\right\rfloor$$ 
and split
\be\label{eq:X++}
X_N=\sum_{0\leq j \leq j_t} X_{j,N} 
 +\sum_{j>j_t} X_{j,N} \,.
\ee

Since by definition of $\tilde \g_{s,N}$ we have $\|P_Nu\|_{L^2} \leq R$ on a set of full $\tilde \g_{s,N}$ measure, the following bound 
$$
\sum_{0\leq j \leq j_t} X_{j,N}  \leq \sum_{0 \leq j \leq  j_t} 2^{j(s-\frac12)}\|\D_j P_N u\|_{L^2}\leq R2^{j_t\left( s- \frac12 \right)} \leq t^{\frac12-\frac{1}{4s}}\,,
$$
holds $\tilde \g_{s,N}$-a.s., therefore
\be\label{eq:X-e1}
\tilde\g_{s,N}\Big(\sum_{0 \leq j \leq j_t} X_{j,N} \geq t^{\frac12-\frac{1}{4s}}\Big)=0\,.
\ee

Now we analyse the second summand of (\ref{eq:X++}). Let $c_0>0$ small enough so that
\be\label{eq:sigma}
\s_j:=c_0j^{-l}\,,\quad \sum_{j\in\N}\s_j\leq1 \,.
\ee
For any $j\in\N$ we have
\be\nonumber %\label{eq:sub-exp-toprove}
\tilde\g_{s,N}( X_{j,N} \geq \s_jt^{\frac12-\frac{1}{4s}})
=  \tilde\g_{s,N}( 2^{j(s-\frac12)}\|\D_jP_Nu\|_{L^2} \geq \s_jt^{\frac12-\frac{1}{4s}})\,.
\ee
By Bernstein inequality (\ref{eq_Bernstein}) 
\bea
\tilde\g_{s,N}( 2^{j(s-\frac12)}\|\D_jP_Nu\|_{L^2} \geq \s_jt^{\frac12-\frac{1}{4s}})&=&\tilde\g_{s,N}( \|\D_jP_Nu\|^2_{L^2} \geq 2^{-2j(s-\frac12)}\s^2_jt^{2\frac{2s-1}{4s}})\nn\\
&\leq& C\exp\left(-c2^j\min\left(\s_j^2t^{\frac{2s-1}{2s}},\s_j^4t^{\frac{2s-1}{s}}\,\right)\right)\,.
\eea
Therefore for any fixed $j$ we have
$$
\tilde\g_{s,N}( 2^{j(s-\frac12)}\|\D_jP_Nu\|_{L^2} \geq \s_jt^{\frac{1}{2} - \frac{1}{4s}})\leq Ce^{-\s_j^2 t^{\frac{2s-1}{2s}} 2^{-j}}
$$
provided 
$$
t\gtrsim j^{\frac{4s}{2s-1} l}\,. 
$$
Therefore the estimate extends to all $j\leq \lceil\log_2N\rceil$ for 
$$
t\gtrsim (\log_2N)^{\frac{4s}{2s-1} l}\,. 
$$ 
Then 
\bea
\tilde\g_{s,N}\Big(\sum_{j > j_t} X_{j,N} \geq t^{ \frac{1}{2} - \frac{1}{4s} }\Big) 
&\leq& \sum_{j > j_t} \tilde\g_{s,N}\Big(X_{j,N} \geq \s_jt^{ \frac{1}{2} - \frac{1}{4s} }\Big)
\nn\\
&\lesssim& \sum_{j > j_t} e^{-2^j\s_j^2t^{\frac{2s-1}{2s}}}\leq C\exp\left(-c\frac{t}{R^{\frac{2}{2s-1}}}\right)\,\nn
\eea
for some absolute constants $C,c>0$. 
\end{proof}

\begin{lemma}\label{lemma:subexpY}
There is $c>0$ such that
\be
\tilde\g_{s,N}(Y_N\geq t^{\frac{1}{4s}})\leq \exp\left(-\frac{ct}{R^{2(2s-1)}}\right)\,. 
\ee
\end{lemma}
\begin{proof}
We set now $$j_t:=\left\lfloor\log_2\frac{t^{\frac{1}{2s}}}{R^2}\right\rfloor\,.$$ Since $\|P_Nu\|_{L^2} \leq R$ on a set of full $\tilde \g_{s,N}$ measure, the simple inequality
$$
Y_{j,N}\leq 2^{\frac j2}\|\D_jP_N\|_{L^2}
$$
implies the following bound to hold $\tilde \g_{s,N}$-a.s.
\be
\sum_{0 \leq j \leq j_t} Y_{j,N}\leq \sum_{0 \leq j \leq j_t} 2^{\frac j2}\|\D_j P_N u\|_{L^2}\leq R2^{\frac{j_t}{2}}=t^{\frac{1}{4s}}\,.
\ee
Thus
\be\label{eq:L-e1}
\tilde\g_{s,N}\Big(\sum_{0 \leq j \leq j_t} Y_{j,N}\geq t^{\frac{1}{4s}}\Big)=0\,.
\ee

To estimate the contribution for $j>j_t$ we consider again the $\s_{j}>0$ defined in (\ref{eq:sigma}) and bound
\be\label{eq:L-interme}
\tilde\g_{s,N}\Big(\sum_{j>j_t}Y_{j,N}\geq t^{\frac{1}{4s}}\Big)\leq \sum_{j>j_t}\tilde\g_{s,N}\Big(Y_{j,N}\geq \s_jt^{\frac{1}{4s}}\Big)\,.
\ee
We have the following estimate for $Y_{j,N}$ from Hoeffding inequality
\be
\tilde\g_{s,N}(Y_{j,N}\geq t)\leq C\exp\left(-\frac{ct^2}{\sum_{|n|\simeq 2^j}\meanv{n}^{-2s}}\right)\,
\ee
for suitable constants $C,c>0$. 
Therefore
\be\label{eq:L-tail}
\tilde\g_{s,N}\left(Y_{j,N}\geq \s_jt^{\frac{1}{4s}}\right)\lesssim e^{-c\s_j^22^{j(2s-1)}t^{\frac{1}{2s}}}\quad\mbox{for some $c>0$}\,. 
\ee
Thus, noting
$$
2^{j_t(2s-1)}t^{\frac{1}{2s}} =\frac{t}{R^{2(2s-1)}} \, ,
$$
there is some $C,c>0$ such that we can bound 
\be
\mbox{r.h.s. of (\ref{eq:L-interme})}\lesssim\sum_{j>j_t }e^{-c\s_j^2 2^{j(2s-1)}t^{\frac{1}{2s}}}\leq C\exp\left(-\frac{ct}{R^{2(2s-1)}}\right)\,,
\ee
that concludes the proof.
\end{proof}

Combining (\ref{eq:union-bound-XY}) with Lemma \ref{lemma:subexpX} and Lemma \ref{lemma:subexpY} we prove Proposition \ref{prop:sub-exp1}. Finally the bulk of the distribution of $F$ is easier to bound.

\begin{proof}[Proof of Proposition \ref{prop:sub-exp2}]
Let us set $T:=\lfloor t^{\frac{2}{2s-1}}\rfloor$. Notice that since we restrict to $t\leq \sqrt{N^{2s-1}}$ we have $N > T$. We use the union bound
\be\label{eq:subexp2-use}
\tilde\g_{s,N}(|F_N|\geq t)\leq \g_{s}(|F_N-F_T|\geq t/2)+\tilde\g_{s,N}(|F_T|\geq t/2)\,. 
\ee
By Proposition \ref{prop:exp0} we have
\be
\g_{s}(|F_N-F_T|\geq t)\leq Ce^{-c\sqrt t T^{\frac{2s-1}{4}}}\leq Ce^{-ct}\,. 
\ee
On the other hand since $t > \sqrt{T^{2s-1}}$ the estimate (\ref{eq:sub-exp1}) applies to the second summand of \eqref{eq:subexp2-use}. This concludes the proof. 
\end{proof}

\begin{proof}[Proof of Lemma \ref{lemma:L-Pbound}]
The first step is to revisit the decomposition of Lemma \ref{Lemma:decomposizione}. We set
\bea
F_N^{\geq}&:=&\sum_{\substack{j_1,\ell_1\in\N\\J\geq \min(j_1,\ell_1)}} \Im \int \left( \D_{j_1}P_N\bar u\right)^{(2s)} \left( \D_{\ell_1}P_N u \right) \left( \D_{J}\mc I[P_Nu] \right)\,,\label{eq:tildeFgeq}\\
F_N^{<}&:=&\sum_{\substack{j_1,\ell_1\in\N\\J<\min(j_1,\ell_1)}} \Im \int \left( \D_{j_1}P_N\bar u\right)^{(2s)} \left( \D_{\ell_1}P_N u \right) \left( \D_{J}\mc I[P_Nu] \right)\,.\label{eq:tildeF<}
\eea
Of course it is
\be\label{eq:dec2}
F_N= F_N^{\geq}+ F_N^{<}\,.
\ee
With a slight abuse of notation we denote these two quantities by the same symbols as the ones of the decomposition of Lemma \ref{Lemma:decomposizione}. This will lighten the exposition and anyway the analogy between the two representations is clear. 
We will also need
\be\label{eq:mcI-L-P}
\D_J \mc I[P_Nu]=i\sum_{j_2,\ell_1\in\N}\sum_{\substack{|n_2|,|m_2|\leq N\\|n_2|\simeq 2^{j_2},\,|m_2|\simeq 2^{\ell_2}\\0\neq|n_2-m_2|\simeq 2^J}}e^{i(n_2-m_2)x}\frac{\bar u(m_2)u(n_2)}{n_2-m_2}\,.
\ee

Using (\ref{eq:mcI-L-P}) we bound

\be\label{eq:F>interm}
|F_N^{\geq}|\lesssim\Big|\sum_{\substack{j_1,j_2,\ell_1,\ell_2\in\N\\J\geq \min(\ell_1,j_1)}}\sum_{\substack{|n_1|,|n_2|,|m_1|,|m_2|\leq N\\|m_1|\simeq 2^{\ell_1},\,|n_1|\simeq 2^{j_1}\\|n_2|\simeq 2^{j_2}\,,|m_2|\simeq 2^{\ell_2}\\0\neq|n_1-m_1|\simeq 2^J\\n_1+n_2=m_1+m_2}}\frac{|m_1|^{2s}}{m_1-n_1}\bar u(m_1)\bar u(m_2)u(n_1)u(n_2)\Big|\,.
\ee
The constraints 
$$
|n_1-m_1|\simeq 2^J,  \quad |n_1| \simeq 2^{j_1}, \quad |m_1| \simeq 2^{\ell_1}
$$
in the inner sum enforces one of the following possibilities:
\begin{equation}
\begin{array}{ll}
\mbox{(A)} & n_1\simeq m_1 \quad \mbox{and} \quad J<\ell_1=j_1\, ,
\\
\mbox{(B)} & J=\ell_1> j_1\, ,
\\
\mbox{(C)} & J=j_1> \ell_1 \, , 
\end{array}
\end{equation}
However (A) is excluded by the condition $J\geq \min(\ell_1,j_1)$ in the outer sum.

$\bullet\,$Case (B).
Noting that in this case we have
$$
\frac{|m_1|^{2s}}{m_1-n_1} \lesssim 2^{(2s-1)\ell_1} \,.
$$
Since the sum is restricted over $n_1+n_2=m_1+m_2$ we can assume that at leas one between 
$n_2, m_2$ is comparable to $m_1$. To fix the notations we will assume this index to be $n_2$, so that 
$$
2^{\ell_1} \simeq |m_1| \simeq 2^{j_2} \simeq |n_2| \gtrsim |n_1|, |m_2| \,.
$$
We have 
\begin{align}\label{FinalRHS1}
(\ref{eq:F>interm})\Big|_{(B)}&\lesssim \sum_{\substack{j_1,j_2,\ell_1,\ell_2\in\N\\\ell_1> j_1}}2^{(2s-1)\ell_1}\sum_{\substack{|n_1|,|n_2|,|m_1|,|m_2|\leq N\\|m_1|\simeq 2^{\ell_1},\,|n_1|\simeq 2^{j_1}\\|n_2|\simeq 2^{j_2}\,,|m_2|\simeq 2^{\ell_2}\\n_1+n_2=m_1+m_2}} | u(m_1)| |  u(m_2)| |u(n_1)| |u(n_2)| \nn\\
&\lesssim \sum_{\substack{j_1,j_2,\ell_1,\ell_2\in\N}}2^{\ell_1(s-\frac12)}2^{j_2(s-\frac12)}\sum_{\substack{|n_1|,|n_2|,|m_1|,|m_2|\leq N\\|m_1|\simeq 2^{\ell_1},\,|n_1|\simeq 2^{j_1}\\|n_2|\simeq 2^{j_2}\,,|m_2|\simeq 2^{\ell_2}\\n_1+n_2=m_1+m_2}} | u(m_1)| | u(m_2)| |u(n_1)| |u(n_2)| \nn\\
\end{align}
Now for fixed $n_1, m_2$ we have, by Cauchy-Schwartz and Plancherel inequality
$$
2^{\ell_1(s-\frac12)}2^{j_2(s-\frac12)}\sum_{\substack{|n_2|,|m_1|\leq N\\|m_1|\simeq 2^{\ell_1}, \, |n_2|\simeq 2^{j_2} \\n_1+n_2=m_1+m_2}}
|u(m_1)| |u(n_2)| \leq 
 2^{\ell_1(s-\frac12)} \|  \Delta_{\ell_1} u \|_{L^2} 2^{j_2(s-\frac12)} \|  \Delta_{j_2}  u\|_{L^2}
$$
Plugging this into the r.h.s. of \eqref{FinalRHS1} we get
\begin{align}
(\ref{eq:F>interm})\Big|_{(B)}
& \lesssim
\left( \sum_{\ell_1 \in \N}  2^{\ell_1(s-\frac12)}   \|  \Delta_{\ell_1} u \|_{L^2} \right) 
\left(  \sum_{j_2 \in \N} 2^{j_2(s-\frac12)} \|  \Delta_{j_2}  u\|_{L^2} \right) 
\\ \nonumber
&  \quad \quad \quad \quad \quad \quad \quad \quad \quad 
\times \sum_{\substack{j_1, \ell_2\in\N}}  \sum_{\substack{|n_1|,|m_2|\leq N\\ |n_1|\simeq 2^{j_1} \,,|m_2|\simeq 2^{\ell_2}}}
| u(m_2)| |u(n_1)| 
\\ \nonumber
& 
\lesssim
X_N^2 
\left( \sum_{j_1 \in\N}  \sum_{\substack{|n_1|\simeq 2^{j_1} }}
 |u(n_1)|  \right)
\left( \sum_{\ell_2\in\N}  \sum_{ |m_2|\simeq 2^{\ell_2} } 
| u(m_2)|  \right) = X_N^2  Y_N^2
\end{align}

$\bullet\,$Case (C). Since $\ell_1 < j_1$ we have
$$
2^{(2s-1)\ell_1} \leq 2^{(s-\frac12)\ell_1} 2^{(s-\frac12) j_1}  
$$ 
thus
\bea
(\ref{eq:F>interm})\Big|_{(C)}&\lesssim& \sum_{\substack{j_1,j_2,\ell_1,\ell_2\in\N}}2^{(s-\frac12)\ell_1} 2^{(s-\frac12) j_1}  
\sum_{\substack{|n_1|,|n_2|,|m_1|,|m_2|\leq N\\|m_1|\simeq 2^{\ell_1},\,|n_1|\simeq 2^{j_1}\\|n_2|\simeq 2^{j_2}\,,|m_2|\simeq 2^{\ell_2}\\n_1+n_2=m_1+m_2}}|u(m_1)| |u(m_2)| |u(n_1)| |u(n_2)| \nn\\
\eea
and proceed as in the case (B), switching $n_1 \leftrightarrow n_2$. 

It remains to bound 
$$
|F^{<}_N| = \Big|\Im\Big(\sum_{\substack{j_1,j_2,\ell_1,\ell_2\in\N\\J< \min(\ell_1,j_1)}}\sum_{\substack{|n_1|,|n_2|,|m_1|,|m_2|\leq N\\|m_1|\simeq 2^{\ell_1},\,|n_1|\simeq 2^{j_1}\\|n_2|\simeq 2^{j_2}\,,|m_2|\simeq 2^{\ell_2}\\0\neq|n_1-m_1|\simeq 2^J\\n_1+n_2=m_1+m_2}}\frac{|m_1|^{2s}}{m_1-n_1}\bar u(m_1)\bar u(m_2)u(n_1)u(n_2)\Big)\Big|.
$$
Since $J< \min(\ell_1,j_1)$ we are summing over $n_1\neq0$. 
Using the fractional binomial identity (recall $n_2 - m_2 = m_1-n_1$)
\begin{align}
\frac{|m_1|^{2s}}{m_1-n_1} 
& = \frac{|m_1|^{s}|n_1+n_2-m_2|^s}{m_1-n_1}
\\ \nonumber
&
= \frac{|m_1|^s |n_1|^{s}\left(1+\frac{n_2-m_2}{|n_1|}\right)^s}{m_1-n_1}
= \sum_{k \geq 0}\frac{(s)_k}{k!} \frac{|m_1|^s |n_1|^{s}(m_1-n_1)^{k-1}}{|n_1|^k}
\end{align}
and the cancellation already exploited in Lemma~\ref{Lemma:decomposizione} (see formula (\ref{eq:canc})) 
\be\label{eq:canc2}
\Im\Big(\sum_{\substack{j_1, j_2, \ell_1, \ell_2 \in\N\\J< \min(\ell_1,j_1)}}\sum_{\substack{|n_1|,|n_2|,|m_1|,|m_2|\leq N\\|m_1|\simeq 2^{\ell_1},\,|n_1|\simeq 2^{j_1}\\|n_2|\simeq 2^{j_2}\,,|m_2|\simeq 2^{\ell_2}\\0\neq|n_1-m_1|\simeq 2^J\\n_1+n_2=m_1+m_2}}\frac{|m_1|^s |n_1|^{s}}{m_1-n_1}\bar u(m_1)\bar u(m_2)u(n_1)u(n_2)\Big)=0\,,
\ee
we arrive to
\begin{equation}
|F^{<}_N|
\lesssim 
\Big|\sum_{k\geq1}\frac{(s)_k}{k!}\Im\Big(\sum_{\substack{j_1,j_2,\ell_1,\ell_2\in\N\\J< \min(\ell_1,j_1)}}\sum_{\substack{|n_1|,|n_2|,|m_1|,|m_2|\leq N\\|m_1|\simeq 2^{\ell_1},\,|n_1|\simeq 2^{j_1}\\|n_2|\simeq 2^{j_2}\,,|m_2|\simeq 2^{\ell_2}\\0\neq|n_1-m_1|\simeq 2^J\\n_1+n_2=m_1+m_2}}\frac{|m_1|^s |n_1|^{s}(m_1-n_1)^{k-1}}{|n_1|^k}\bar u(m_1)\bar u(m_2)u(n_1)u(n_2)\Big)\Big|\nn\,\\\label{eq:last}\,.
\end{equation}
Since $J< \min(\ell_1,j_1)$ implies $|n_1-m_1| \leq \frac12 |n_1|$ and that $n_1$ and $m_1$ are comparable, bringing the modulus inside and summing over $k$
we arrive to
\begin{equation}
|F^{<}_N|
 \lesssim 
 \sum_{\substack{j_1,j_2,\ell_1,\ell_2\in\N}}2^{(s-\frac12)\ell_1} 2^{(s-\frac12) j_2}  
\sum_{\substack{|n_1|,|n_2|,|m_1|,|m_2|\leq N\\|m_1|\simeq 2^{\ell_1},\,|n_1|\simeq 2^{j_1}\\|n_2|\simeq 2^{j_2}\,,|m_2|\simeq 2^{\ell_2}\\n_1+n_2=m_1+m_2}}|u(m_1)| |u(m_2)| |u(n_1)| |u(n_2)|
\end{equation}
and we proceed as in the case (B).
\end{proof}

%%%%%%%%%%% REFERENCES %%%%%%%%%%%%%%%%%%%%%

\end{document}